\documentclass[11pt]{amsart}
\usepackage{amssymb, amstext, amscd, amsmath, amssymb}
\usepackage{mathtools, paralist, color, dsfont, rotating}
\usepackage{verbatim}
\usepackage{enumerate,enumitem}
\usepackage{tikz-cd}
\usepackage{float}
\usepackage{setspace}
\usepackage{todonotes}

\usepackage{hyperref}
\usepackage{tikz}

\floatstyle{boxed} 
\restylefloat{figure}

\numberwithin{equation}{section}
\usepackage[a4paper]{geometry}
\usepackage{quoting}
\quotingsetup{vskip=.1in}
\quotingsetup{leftmargin=.17in}
\quotingsetup{rightmargin=.17in}

\let\OLDthebibliography\thebibliography
\renewcommand\thebibliography[1]{
  \OLDthebibliography{#1}
  \setlength{\parskip}{0pt}
  \setlength{\itemsep}{2pt plus 0.5ex}
}

\makeatletter
\def\@cite#1#2{{\m@th\upshape\bfseries%
[{#1\if@tempswa{\m@th\upshape\mdseries, #2}\fi}]}}
\makeatother
\theoremstyle{plain}
\newtheorem{theorem}{Theorem}[section]
\newtheorem{corollary}[theorem]{Corollary}
\newtheorem{proposition}[theorem]{Proposition}
\newtheorem{lemma}[theorem]{Lemma}
\theoremstyle{definition}
\newtheorem{definition}[theorem]{Definition}
\newtheorem{example}[theorem]{Example}

\newtheorem{remark}[theorem]{Remark}

\newtheorem{question}[theorem]{Question}

\theoremstyle{remark}

\mathtoolsset{centercolon}
\newcommand{\bbC}{{\mathbb{C}}}
\newcommand{\bbD}{{\mathbb{D}}}

\newcommand{\bbT}{{\mathbb{T}}}

\newcommand{\A}{{\mathcal{A}}}
\newcommand{\B}{{\mathcal{B}}}
\newcommand{\C}{{\mathcal{C}}}
\newcommand{\D}{{\mathcal{D}}}

\newcommand{\I}{{\mathcal{I}}}
\newcommand{\J}{{\mathcal{J}}}
\newcommand{\K}{{\mathcal{K}}}

\newcommand{\M}{{\mathcal{M}}}

\newcommand{\T}{{\mathcal{T}}}

\newcommand{\fA}{{\mathcal{A}}}
\newcommand{\fB}{{\mathcal{B}}}
\newcommand{\fC}{{\mathcal{C}}}
\newcommand{\fD}{{\mathcal{D}}}

\renewcommand{\phi}{\varphi}
\newcommand{\upchi}{{\raise.35ex\hbox{\ensuremath{\chi}}}}

\newcommand{\Aut}{\operatorname{Aut}}

\newcommand{\id}{{\operatorname{id}}}

\newcommand{\spn}{\operatorname{span}}

\newcommand\Span{\mathop{\rm span}}
\newcommand\cspan{\overline{\Span}}

\newcommand\ad{\mathop{\rm ad}}

\newcommand{\cstarlattice}{\text{C$^*$-Lat}}
\newcommand{\Cmax}{C^\ast_\text{max}}

\begin{document}
\title{The lattice of C$^*$-covers of an operator algebra}

\author{Adam~Humeniuk}
\address {Department of Mathematics and Computing
\\Mount Royal University \\ Calgary, AB \\Canada}
\email{ahumeniuk@mtroyal.ca}
\author{Christopher~Ramsey}
\address {Department of Mathematics and Statistics
\\MacEwan University \\ Edmonton, AB \\Canada}
\email{ramseyc5@macewan.ca}

\begin{abstract}
In this paper it is shown that the lattice of C$^*$-covers of an operator algebra does not contain enough information to distinguish operator algebras up to completely isometric isomorphism. In addition, four natural equivalences of the lattice of C$^*$-covers are developed and proven to be distinct. The lattice of C$^*$-covers of direct sums and tensor products are studied. Along the way key examples are found of operator algebras, each of which generates exactly $n$ C$^*$-algebras up to $*$-isomorphism, and a simple operator algebra that is not similar to a C$^*$-algebra. 
\end{abstract}


\maketitle

\section{Introduction}
How much information is lost by passing to the C$^*$-algebra of a completely isometric representation of a non-selfadjoint operator algebra? Plenty, as many non-isomorphic operator algebras generate the same C$^*$-algebra. What about considering the collection of C$^*$-algebras generated by all completely isometric representations of an operator algebra? 
In other words, what does it mean for two operator algebras to generate the ``same'' C$^*$-algebras? 

A C$^*$-algebra generated by a completely isometric representation of an operator algebra is called a C$^*$-cover. Proving the existence of the C$^*$-envelope, the minimal C$^*$-cover, and understanding its structure, in general and in many, many particular algebras, has been a central focus of non-selfadjoint operator algebras. The existence of the C$^*$-envelope was conjectured, and proven in many cases, by Arveson in 1969 \cite{Arv1}, and proven in general by Hamana \cite{Hamana}. On the other hand, the maximal C$^*$-cover simply exists by a standard universality argument over all completely contractive representations and was first studied in detail by Blecher in 1999 \cite{Blech}.

Of course, many C$^*$-covers have been studied in the case when they are generated by a ``natural'' representation of the operator algebra. For instance, $C(\overline{\mathbb D})$ and the Toeplitz algebra are generated by representations of $A(\mathbb D)$, or more generally the Toeplitz-Cuntz-Pimsner algebra is a C$^*$-cover of a tensor algebra of a C$^*$-correspondence, which plays an important role in the theory \cite{MS1}. This is not in the least surprising since the most tractable operator algebras are those with nice concrete representations. 

C$^*$-covers and their morphisms were first outlined in Blecher and Le Merdy's book \cite[Chapter 2]{Blecher} and they consider that Paul Muhly came up with the name. In \cite{KatRamMem} the second author and Katsoulis used the structure of C$^*$-covers to develop the theory of non-selfadjoint crossed products, and in \cite{Humen} the first author studied C$^*$-covers of semicrossed products, building on earlier results of Davidson, Fuller, and Kakariadis \cite{DFK}. The complete lattice structure was proven by Hamidi \cite{Hamidi} and Thompson \cite{Thompson} and the latter associates this lattice to the spectrum of the maximal C$^*$-cover. It was shown in \cite{KatRamHN} and \cite{Hamidi} that not every C$^*$-cover does a good job of encoding the structure of the operator algebra as there are completely isometric automorphisms of the operator algebra which fail to extend to $*$-automorphisms of some of the C$^*$-covers. It should be noted that the C$^*$-envelope and the maximal C$^*$-cover never fail this extension by way of their universal properties.

So when are the C$^*$-covers of two operator algebras the ``same''? In Section 2 (Definition \ref{def:sameness}) we develop four possible concepts of ``same-ness": \textit{lattice intertwined, lattice $*$-isomorphic, lattice isomorphic, and C$^*$-cover equivalent}. All four equivalences are shown to be distinct from each other and from completely isometric isomorphism by Theorem \ref{thm:allrelationsdifferent}, implying that the lattice of C$^*$-covers does not necessarily give enough information to tell apart two non-completely-isometrically-isomorphic operator algebras.

On the way to the main goal of the paper, we provide a significant amount of theory around these equivalences with many examples. In particular, the lattice of C$^*$-covers of a direct sum where at least one operator algebra is approximately unital is the product of the lattices, Theorem \ref{thm:direct_sum_covers}. Studying tensor products and C$^*$-covers is known to be enigmatic \cite[Section 6.2]{Blecher} but we can describe the lattice of C$^*$-covers of tensor product of an operator algebra with a simple, nuclear C$^*$-algebra, Theorem \ref{thm:max_tensor_lattice}, along with other partial results. 

In Section 3 it is asked whether there can be an operator algebra whose C$^*$-envelope is also its maximal C$^*$-cover, akin to the operator system found by Kirchberg and Wassermann \cite{KirWas} that enjoys this property. Negative results in some cases are developed. More interesting is the fact that for every integer $n$ there is an operator algebra that generates exactly $n$ C$^*$-algebras up to $*$-isomorphism, Theorem \ref{thm:ncstarcovers}, but do note that this is not up to C$^*$-cover isomorphism so the lattice of C$^*$-covers need not be finite.

Finally, Section 4 discusses several concepts - RFD, action admissibility, and simplicity - and their relationship to the lattice of C$^*$-covers and its equivalences. Of note is Theorem \ref{thm:simple_not_cb_iso_to_C*}, that there is a simple operator algebra that is not completely boundedly isomorphic to a C$^*$-algebra.

\section{Equivalences of the lattice of C*-covers}

First, we need to recall the complete lattice structure of C$^*$-covers developed in \cite[Chapter 2]{Blecher}, \cite[Section 2.1]{Hamidi} and \cite{Thompson}. For more in-depth background on C*-envelopes and boundary ideals, we recommend \cite{Arv1} and \cite[Chapter 15]{Paulsen}.

\begin{definition}
Let $\A$ be an operator algebra. A {\bf C$^*$-cover} of $\A$ is a pair $(\fC, \iota)$ consisting of a C$^*$-algebra $\fC$ and a completely isometric linear map $\iota : \A \rightarrow \fC$ such that $C^*(\iota(\A)) = \fC$.
\end{definition}

\begin{definition}
If $\A$ is an operator algebra, and $(\fB,\iota)$, $(\fC,\eta)$ are C*-covers of $\A$, then a \textbf{morphism of C*-covers} $\pi:(\fB,\iota)\to (\fC,\eta)$ is a $\ast$-homomorphism $\pi:\fB\to \fC$ which satisfies $\pi\iota=\eta$. If there exists a morphism of C$^*$-covers $(\fB,\iota)\to (\fC,\eta)$, then we write $(\fC,\eta)\preceq (\fB,\iota)$.
\end{definition}

Because the C*-covers satisfy $\fB=C^\ast(\iota(\A))$ and $\fC=C^\ast(\eta(\A))$, it is automatic that if a morphism $\pi:(\fB,\iota)\to (\fC,\eta)$ exists, then the $\ast$-homomorphism $\pi:\fB\to \fC$ is surjective and unique with this property. Consequently, if $(\fB,\iota)\preceq (\fC,\eta)$ and $(\fB,\iota)\succeq (\fC,\eta)$, then the unique morphism $\pi:(\fB,\iota)\to (\fC,\eta)$ consists of a $\ast$-isomorphism $\pi:\fB\to \fC$ with $\pi\iota=\eta$. In this case, we write $(\fB,\iota)\sim (\fC,\eta)$, and $\sim$ defines an equivalence relation on the class of C*-covers of $\A$.  It should be noted that the non-unital case poses no problems here thanks to the work of Meyer \cite{Meyer}, that every completely contractive (isometric) homomorphism of a non-unital operator algebra extends uniquely to a completely contractive (isometric) homomorphism on the unitization.

\begin{definition}
Let $\A$ be an operator algebra. We let $\cstarlattice(\A)$ denote the collection of equivalence classes $[\fB,\iota]$ of C*-covers $(\fB,\iota)$ of $\A$.
\end{definition}

Do note that in \cite{Thompson} $\cstarlattice(\A)$ is called $\operatorname{Cov}(\A)$ but we feel that our naming convention is more natural. Note as well that we will usually drop the ``complete'' in favour of just saying the lattice of C$^*$-covers.

For any operator algebra $\A$, the collection $\cstarlattice(\A)$ is actually a set, and not just a proper class. To see this, one can identify the lattice of C*-covers with the set of boundary ideals for $\A$ in $\Cmax(\A)$ in an order reversing fashion. If $\mu:\A\to \Cmax(\A)$ is the embedding of $\A$ into its maximal C*-algebra, then $[\Cmax(\A),\mu]$ is the maximum element in $\cstarlattice(\A)$. Likewise, if $\epsilon:\A\to C^*_e(\A)$ is the embedding of $\A$ into its C*-envelope, then $[C^*_e(\A),\epsilon]$ is the minimum element of $\cstarlattice(\A)$.

\begin{proposition}[{\cite[Theorem 2.1.6 and Theorem 2.1.11]{Hamidi}} and {\cite[Section 3]{Thompson}}]\label{prop:cstarlattice}
If $\A$ is an operator algebra, then the ordering $\preceq$ on $\cstarlattice(\A)$ makes $\cstarlattice(\A)$ into a complete lattice. Given an arbitrary family $(\fB_\lambda,\iota_\lambda)$ of C*-covers for $\A$, their join is
\begin{align*}
\bigvee_\lambda [\fB_\lambda,\iota_\lambda] &=
\left[C^\ast\left(\left(\bigoplus_\lambda \iota_\lambda\right)(\A)\right),\bigoplus_\lambda \iota_\lambda\right].
\end{align*}
To describe their meet, it is easiest to use the maximal C*-algebra $(\Cmax(\A),\mu)$. Let $\pi_\lambda:(\Cmax(\A),\mu)\to (\fB_\lambda,\iota_\lambda)$ be morphisms of C*-covers, and set $J_\lambda=\ker(\pi_\lambda)\triangleleft \Cmax(\A)$. Then, let
\[
J=
\bigvee_\lambda J_\lambda = \overline{\sum_\lambda J_\lambda}\triangleleft \Cmax(\A),
\]
and let $q:\Cmax(\A)\to \Cmax(\A)/J$ be the quotient map. Then
\[
\bigwedge_\lambda [\fB_\lambda,\iota_\lambda] = 
[\Cmax(\A)/J,\ q\mu].
\]
\end{proposition}

Alternatively, one can avoid enlarging all the way to the maximal C*-algebra when defining the meet in Proposition \ref{prop:cstarlattice}. Given a family $(\fB_\lambda,\iota_\lambda)$ of C*-covers, one can also take the quotient of the C*-algebra $\bigvee_\lambda \fB_\lambda = C^\ast((\oplus_\lambda \iota_\lambda)(\A))$ by the ideal
\[
\overline{\sum_\lambda \ker\left(\bigvee_\mu [\fB_\mu,\iota_\mu]\to [\fB_\lambda,\iota_\lambda]\right)}
\]
generated by the kernel of the C*-cover morphisms from the join to each C*-cover.

For downward directed collections, we can describe the meet without reference to the maximal C*-algebra.

\begin{proposition}\label{prop:meet_colimit}
Let $\A$ be an operator algebra. Let $((\fB_\lambda,\iota_\lambda))_{\lambda\in \Lambda}$ be a subset of $\cstarlattice(\A)$ indexed via a downward directed set $\Lambda$, that is, $[\fB_\lambda,\iota_\lambda]\succeq [\fB_\mu,\iota_\mu]$ if and only if $\lambda\le \mu$. Then,
\[
\bigwedge_\lambda [\fB_\lambda,\iota_\lambda] =
\left[\varinjlim_\lambda \fB_\lambda,\iota\right],
\]
where the direct limit is taken along the unique morphisms
\[
\pi_{\lambda,\mu}:(\fB_\lambda,\iota_\lambda)\to (\fB_\mu,\iota_\mu),
\]
for $\lambda\le \mu$, and if $\fB=\varinjlim_{\lambda}\fB_\lambda$, and $\eta_\lambda:\fB_\lambda\to \fB$ are the universal maps to the direct limit, then $\iota=\eta_\lambda\iota_\lambda$ for all $\lambda\in \Lambda$.
\end{proposition}

\begin{proof}
For $\lambda\le \mu$, the diagram
\[\begin{tikzcd}
    & \fB & \\
    \fB_\lambda \arrow[ur,"\eta_\mu"] \arrow[rr,near start,"\pi_{\lambda,\mu}"] & & \fB_\mu \arrow[ul,"\eta_\pi",swap]\\
    & \A\arrow[ul,hook,"\iota_\lambda"] \arrow[ur,hook,swap,"\iota_\mu"] \arrow[uu,dashed],
\end{tikzcd}\]
commutes, where the vertical map defines $\iota$. For $n\ge 1$ and $a\in M_n(\A)$, for any fixed $\lambda\in \Lambda$, we have
\[
\|\iota^{(n)}(a)\| =
\limsup_{\mu\ge \lambda} \|\pi_{\lambda,\mu}^{(n)}\iota_\lambda^{(n)}(a)\| =
\limsup_{\mu\ge \lambda}\|\iota_\mu^{(n)}(a)\|=\|a\|.
\]
So, $\iota$ is completely isometric. Then, commutation of the diagram implies that $(\fB,\iota)$ is a C*-cover and $[\fB,\iota]\preceq [\fB_\lambda,\iota_\lambda]$ for every $\lambda$. Finally, if $(\fC,\rho)$ is any C*-cover for $\A$ with $[\fC,\rho]\preceq [\fB_\lambda,\iota_\lambda]$ for all $\lambda$, there are unique morphisms $\sigma_\lambda:(\fB_\lambda,\iota_\lambda)\to (\fC,\rho)$. The universal property of $\fB=\varinjlim_{\lambda} \fB_\lambda$ implies that there is a $\ast$-homomorphism $\sigma:\fB\to \fC$ with $\sigma\eta_\lambda=\sigma_\lambda$ for all $\lambda$. Then, $\sigma\iota=\sigma\eta_\lambda\iota_\lambda=\sigma_\lambda\iota_\lambda=\rho$, so $\sigma$ is a morphism of C*-covers and $[\fB,\iota]\succeq [\fC,\rho]$. Therefore, $[\fB,\iota]$ is the greatest lower bound.
\end{proof}

Our main question of interest is: To what extent can two operator algebras have the same lattice of C*-covers? We will discuss four possible notions of ``the same C*-covers".

\begin{definition}\label{def:sameness}
Let $\A$ and $\B$ be operator algebras.
\begin{itemize}
\item [(1)] We say that $\A$ and $\B$ are \textbf{lattice isomorphic} if there is an order isomorphism $F:\cstarlattice(\A)\to\cstarlattice(\B)$.
\item [(2)] We say that $\A$ and $\B$ are \textbf{lattice $\ast$-isomorphic} if there is an order isomorphism $F:\cstarlattice(\A)\to \cstarlattice(\B)$ such that for each C*-cover $(\fC,\iota)$ of $\A$, if $F([\fC,\iota])=[\fD,\eta]$, then $\fD\cong \fC$ as C*-algebras.
\item [(3)] We say that $\A$ and $\B$ are \textbf{lattice intertwined} if there is an order isomorphism $F:\cstarlattice(\A)\to \cstarlattice(\B)$ such that whenever $(\fC,\iota),(\fD,\eta)$ are C*-covers of $\A$ with $(\fD,\eta)\preceq (\fC,\iota)$, there are $\ast$-isomorphisms $\pi_\fC:\fC\to F(\fC)$ and $\pi_\fD:\fD\to F(\fD)$ such that the diagram
\[\begin{tikzcd}
    \fC \arrow[d] \arrow[r,"\pi_\fC"] & F(\fC) \arrow[d] \\
    \fD \arrow[r,"\pi_\fD"] & F(\fD)
\end{tikzcd}\]
commutes, where the vertical maps are the unique morphisms of C*-covers of $\A$ or $\B$, as appropriate.
\item [(4)] We say that $\A$ and $\B$ are \textbf{C*-cover equivalent} if whenever $(\fC,\iota)$ is a C*-cover of $\A$, then there is a C*-cover $(\fD,\eta)$ of $\B$ with $\fC\cong \fD$ as C*-algebras, and conversely if $(\fD,\eta)$ is a C*-cover of $\B$, then $\fD\cong\fC$ for some C*-cover $(\fC,\iota)$ of $\A$.
\end{itemize}
\end{definition}

It is immediate to see we have the following Hasse diagram of equivalence relations between operator algebras
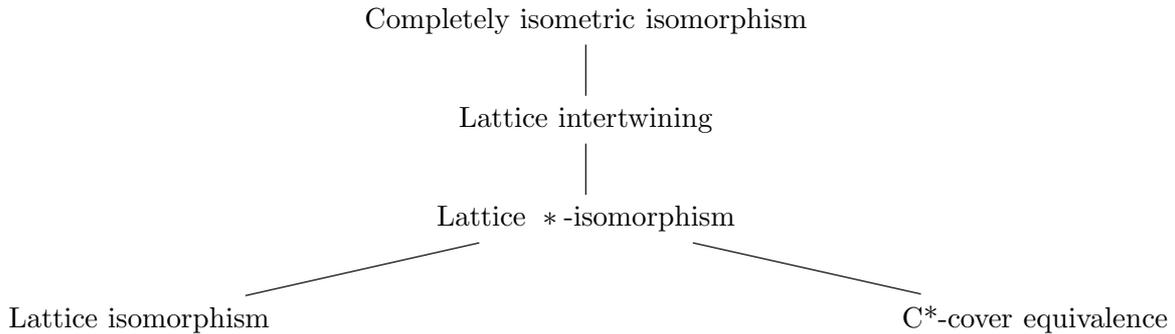
\[\begin{tikzcd}
    & \text{Completely isometric isomorphism} \arrow[d,no head] \\
    &\text{Lattice intertwining} \arrow[d,no head]\\
    &\text{Lattice }\ast\text{-isomorphism} \arrow[dl,no head]\arrow[dr,no head]\\
    \text{Lattice isomorphism} &&\text{C*-cover equivalence}
\end{tikzcd}\]

Our first order of business is to develop some theory around these equivalences so that we can eventually show that these are, in fact, different relations, Theorem \ref{thm:allrelationsdifferent}.

Indeed, it is easy to see that lattice isomorphic is not lattice $*$-isomorphic or C$^*$-cover equivalent.
In particular, if $\fA$ is itself a C*-algebra, then a completely contractive homomorphism of $\fA$ into $B(H)$ is a $\ast$-homomorphism. Consequently, the only C*-cover of $\fA$ is $(\fA,\id_\fA)$. Therefore the lattice $\cstarlattice(\fA)$ is a single point. So, if $\fA$ and $\fB$ are nonisomorphic C*-algebras, then $\fA$ and $\fB$ are lattice isomorphic, but not lattice $\ast$-isomorphic and not C*-cover equivalent.

We next turn to some equivalent statements for lattice intertwined and it should be noted that lattice intertwined is really the most information you can get out of the lattice of C$^*$-covers without knowledge of the specific operator algebras.

\begin{proposition}\label{prop:intertwining}
Suppose $\A$ and $\B$ are operator algebras. The following are equivalent.
\begin{itemize}
\item[(i)] $\A$ and $\B$ are lattice intertwined.
\item[(ii)] There is a $\ast$-isomorphism
\[
\pi:\Cmax(\A)\to \Cmax(\B)
\]
which respects boundary ideals in the sense that $J\triangleleft \A$ is a boundary ideal for $\A$ if and only if $\pi(J)$ is a boundary ideal for $\B$.
\item[(iii)] There exist $*$-isomorphisms $\Phi : \Cmax(\A) \rightarrow \Cmax(\B)$ and $\phi : C^*_e(\A) \rightarrow C^*_e(\B)$ such that the following diagram commutes
\[\begin{tikzcd}
    \Cmax(\A) \arrow[d, "q_\A"] \arrow[r,"\Phi"] & \Cmax(\B) \arrow[d, "q_\B"] \\
    C^*_e(\A) \arrow[r,"\phi"] & C^*_e(\B)
\end{tikzcd}\]
where $q_\A$ and $q_\B$ are the unique canonical quotient maps.
\end{itemize}
\end{proposition}
\begin{proof}
First let us see that (i) and (ii) are equivalent. 
As it is a poset, the lattice $\cstarlattice(\A)$ is a category. We can define a contravariant functor $G_\A$ from the category $\cstarlattice(\A)$ that sends each class to a representative C*-algebra in a C*-cover, and whenever $(\fD,\eta)\preceq (\fC,\iota)$, the associated morphism $[\fD,\eta]\preceq [\fC,\iota]$ is sent to the unique $\ast$-homomorphism $\fC\to \fD$ intertwining $\iota$ and $\eta$. Then, lattice intertwining is just the requirement that the functors $G_\A$ and $G_\B$ are naturally isomorphic.

Every class in $\cstarlattice(\A)$ has a representative of the form $[\Cmax(\A)/J,q_J\mu]$, where $J\triangleleft \Cmax(\A)$ is a boundary ideal for $\mu(\A)$, and the associated morphisms are just the induced quotient maps between boundary ideals. This gives a canonical choice for the functors $G_\A$ and $G_\B$.

(i) implies (iii) is trivial. We will show that (iii) implies (i). Assume now that (iii) is true. 
Suppose $(\fC,\iota)$ is a C$^*$-cover of $\A$. By universality there exist unique $*$-homomorphisms $q_1 : \Cmax(\A) \rightarrow \fC$ with $\iota = q_1\mu_\A$ and $q_2 : \fC \rightarrow C^*_e(\A)$ with $q_2\iota = \epsilon_\A$. 
Define $\fD = \Cmax(\B)/\Phi(\ker q_1)$ and let $q'_1 : \Cmax(\B) \rightarrow \fD$ be the induced quotient $*$-homomorphism. By uniqueness of quotient maps 
\[
\ker q'_1 =\Phi(\ker q_1) \subseteq \Phi(\ker q_2q_1) = \Phi(\ker q_\A) = \ker q_\B.
\]
Thus, $q'_1$ is completely isometric on $\B$ and so $(\fD, q'_1\mu_\B)$ is a C$^*$-cover of $\B$. 
Now define $F([\fC,\iota]) = [\fD, q'_1\mu_B]$ which we will now prove is the lattice intertwining. 

If $(\tilde \fC, \tilde\iota) \sim (\fC, \iota)$ by the $*$-isomorphism $\tilde \varphi$ then we have the following commutative diagram:
\[\begin{tikzcd}
    \Cmax(\A) \arrow[d, "\exists! \tilde q_1",dashed] \arrow[r, "\exists! \tilde\Phi",dashed] & \Cmax(\A) \arrow[d, "q_1"] 
    \\
    \tilde{\fC} \arrow[r, "\tilde\phi"] & \fC 
\end{tikzcd}
\]
It then follows that $\Cmax(\B)/\Phi(\tilde\Phi(\ker\tilde q_1)) = \Cmax(\B)/\Phi(\ker q_1) = \fD$. Hence, $F[(\tilde \fC, \tilde\iota)] = [\fD, q'_1\mu_B]$ is a well-defined map.

Interchanging the roles of $\A$ and $\B$ yields a well-defined function going the other way which acts as the inverse. Thus, $F$ is a bijective function on $\cstarlattice(\A)$. The uniqueness of the quotient maps yields the desired intertwining property (i).
\end{proof}

We will see now that lattice intertwining is not strong enough to distinguish an algebra and its conjugate.

\begin{proposition}\label{prop:intertwinedwithopposite}
If $\A$ be an operator algebra, then $\A$ and $\A^*$ are lattice intertwined.
\end{proposition}
\begin{proof}
Assume that $\A$ is a unital operator algebra. If $\pi:\A\to B(H)$ is a completely contractive representation, then $\pi$ extends uniquely to a completely positive map $\tilde{\pi}:\A+\A^\ast\to B(H)$ by \cite[Proposition 3.5]{Paulsen}. Since a positive map is selfadjoint, $\tilde{\pi}\vert_{\A^\ast}$ satisfies $\tilde{\pi}(a^\ast)=\pi(a)^\ast$ for all $a\in \A$, from which it follows that $\tilde{\pi}\vert_{\A^\ast}$ is an algebra homomorphism. Moreover, if $\pi$ is completely isometric, then $\tilde{\pi}$ is a complete isometry, and so $\tilde{\pi}\vert_{\A^\ast}$ is completely isometric.

In particular, this gives the following commutative diagram:
\[\begin{tikzcd}
    \Cmax(\A) \arrow[d, "q"] \arrow[r,"\id"] & \Cmax(\A^*) \arrow[d, "q"] \\
    C^*_e(\A) \arrow[r,"\id"] & C^*_e(\A^*)
\end{tikzcd}\]
It is now straightforward to see that this result extends to the non-unital case by \cite{Meyer}. In particular, comparing universal properties shows that the natural $\ast$-homomorphism $\Cmax(\A)\to \Cmax(\A^1)$ is injective, and as a consequence of Meyer's result, $C^*_e(\A)$ lives naturally as a C*-subalgebra of $C^*_e(\A^1)$, see \cite[Proposition 4.3.5]{Blecher}. Because $(A^\ast)^1=(\A^1)^\ast$, upon unitizing and applying the previous argument, one can just restrict at the end to the C*-algebras generated by $\A$ in $\Cmax(\A)$ and $C^*_e(\A)$. Therefore, by the previous proposition, $\A$ and $\A^*$ are lattice intertwined.
\end{proof}

\begin{example}\label{ex:intertwinednotcii}
However, $\A$ and $\A^\ast$ may not be completely isometrically isomorphic. For instance, let
\[
\A = 
\left\{\begin{pmatrix}
    a & b & c \\
    0 & a & 0 \\
    0 & 0 & a
\end{pmatrix}\;\middle|\;
a,b,c\in \bbC\right\}.
\]
And so 
\[
\A^* = \left\{\begin{pmatrix}
    a & 0 & 0 \\
    b & a & 0 \\
    c & 0 & a
\end{pmatrix}\;\middle|\;
a,b,c\in \bbC\right\}.
\]
Suppose for the sake of contradiction that a completely isometric algebra homomorphism $\rho:\A\to \A^\ast$ exists. If
\[
N =
\spn\{E_{12},E_{13}\}
\]
consists of the nilpotents in $A$, then $N^\ast$ consists of the nilpotents in $A^\ast$. Since $\rho$ is an algebra homomorphism, $\rho$ preserves nilpotency and so $\rho(N)=N^\ast$. Thus $\rho$ restricts to a completely isometric map between $N$ and $N^\ast$. However, no such complete isometry exists, because $N$ is isomorphic to the row Hilbert space $R_2$ of dimension $2$, and $N^\ast$ is isomorphic to column Hilbert space $C_2$, and $R_2$ and $C_2$ are not completely isometric as operator spaces \cite[Chapter 14]{Paulsen}.


Therefore, lattice intertwining and completely isometrically isomorphic are different equivalences.
\end{example}

This is not the only way to have lattice intertwined operator algebras, see Proposition \ref{prop:intertwinedidempotents}. With either example we see that operator algebras can indeed have the ``same'' lattice of C$^*$-algebras while not being completely isometrically isomorphic.

Recall now that an operator algebra $\A$ is \textbf{approximately unital} if it contains a contractive approximate identity. That is, there is a net $(e_i)_{i\in I}$ of contractions in $A$ such that $\lim_i e_ia=\lim_i ae_i=a$ for all $a\in \A$. Every C$^*$-algebra is approximately unital.

\begin{theorem}\label{thm:direct_sum_covers}
Let $\A$ and $\B$ be operator algebras, and suppose that at least one of $\A$ or $\B$ is approximately unital. If $\pi:\A\oplus \B\to B(H)$ is a completely contractive homomorphism, then
\[
C^\ast(\pi(\A\oplus \B)) \cong C^\ast(\pi_\A(\A))\oplus C^\ast(\pi_\B(\B)),
\]
where $\pi_\A(a)=\pi(a, 0)$ and $\pi_\B(b)=\pi(0,b)$ for $a\in \A$ and $b\in \B$. Consequently, there is a lattice isomorphism between
\[
\cstarlattice(\A\oplus \B) \quad  \textrm{and} \quad\cstarlattice(\A)\times\cstarlattice(\B).
\]
\end{theorem}

\begin{proof}
To show that $C^\ast(\pi(\A\oplus \B))\cong C^\ast(\pi_\A(\A))\oplus C^\ast(\pi_\B(\B))$, it suffices to show that 
\[
C^\ast(\pi_\A(\A))\cdot C^\ast(\pi_\B(\B))=\{0\}.
\]
This will follow from a standard approximation argument using $\ast$-polynomials in $\A$ and $\B$ if we show that for all $a\in \A$ and $b\in \B$, $\pi_\A(a)\pi_\B(b)=\pi_\B(b)\pi_\A(a)=\pi_\A(a)^\ast \pi_\B(b)=\pi_\B(b) \pi_\A(a)^\ast = 0$. 

Now, $\pi$ is an algebra homomorphism and so $\pi_\A(a)\pi_\B(b)=\pi_\B(b)\pi_\A(a)=0$.
Without loss of generality, suppose that $\A$ is approximately unital, with approximate unit $(e_i)_{i\in I}$. Then $(\pi_\A(e_i))_{i\in I}$ is an approximate unit for $\pi_\A(\A)$. A standard application of the C*-identity and triangle inequality shows that both $(\pi_\A(e_i)^\ast\pi_\A(e_i))_{i\in I}$ and $(\pi_\A(e_i)\pi_\A(e_i)^\ast)_{i\in I}$ are approximate identities for the C*-algebra $C^\ast(\pi_\A(\A))$, see \cite[Lemma 2.1.7]{Blecher}. Therefore 
\[\pi_\A(a)=\lim_i \pi_\A(e_i)^\ast\pi_\A(e_i)\pi_\A(a),\]
and so
\[
\pi_\A(a)^\ast \pi_\B(b) =
\lim_{i\in I} \pi_\A(a)^\ast \pi_\A(e_i)^\ast \pi_\A(e_i)\pi_\B(b) = 0
\]
because $\pi_\A(e_i)\pi_\B(b) = 0$. A symmetrical argument shows $\pi_\B(b)\pi_\A(a)^\ast=0$, and therefore 
\[
C^\ast(\pi_\A(\A))C^\ast(\pi_\B(\B))=\{0\},
\]
which implies $C^\ast(\pi(\A\oplus \B)) \cong C^\ast(\pi_\A(\A))\oplus C^\ast(\pi_\B(\B))$ as C*-algebras.

Now, in the case when $\pi$ is completely isometric, this shows that every C*-cover for the operator algebra $\A\oplus \B$ must be of the form $(\fC\oplus \fD, \iota_\fC\oplus\iota_\fD)$ where $(\fC,\iota_\fC)$ is a C*-cover of $\A$ and $(\fD,\iota_\fD)$ is a C*-cover of $\B$. Moreover, a C*-cover morphism
\[
\sigma:(\fC_1\oplus \fD_1,\iota_{\fC_1}\oplus \iota_{\fD_1})\to (\fC_2\oplus \fD_2,\iota_{\fC_2}\oplus \iota_{\fD_2})
\]
between two such covers maps $\iota_{\fC_1}(\A)$ to $\iota_{\fC_2}(\A)\subseteq \fC_2\oplus 0$, and therefore satisfies $\sigma(\fC_1\oplus 0) = \sigma(C^\ast(\iota_{\fC_1}(\A)))\subseteq \fC_2\oplus 0$. Similarly $\sigma(0\oplus \fC_1)\subseteq 0\oplus \fC_2$. Therefore $\sigma$ must have the form $\sigma=\sigma_\A\oplus \sigma_\B$ where $\sigma_\A:\fC_1\to \fC_2$ and $\sigma_\B:\fD_1\to \fD_2$ are morphisms of C*-covers. And, any such direct sum of morphisms of C*-covers yields a morphism. This shows that there is a lattice isomorphism between
\[
\cstarlattice(\A\oplus \B) \quad \textrm{and} \quad \cstarlattice(\A)\times\cstarlattice(\B). \qedhere
\]
\end{proof}

\begin{corollary}
Suppose $\A$, $\tilde \A$, $\B$, and $\tilde \B$ are operator algebras with $\A$ and $\tilde \A$ approximately unital. If both pairs $\A$, $\tilde \A$ and $\B$, $\tilde \B$ are individually completely isometrically isomorphic, lattice intertwining, lattice $*$-isomorphic, lattice isomorphic, or C$^*$-cover equivalent then so are $\A\oplus \B$ and $\tilde \A \oplus \tilde \B$.
\end{corollary}

\begin{corollary}\label{cor:addacstaralgebra}
If $\A$ is an operator algebra and $\fC$ is a C$^*$-algebra, then $\A$ and $\A\oplus \fC$ are lattice isomorphic.
\end{corollary}
\begin{proof}
The previous theorem shows that $\A$ and $\A\oplus \fC$ are naturally lattice isomorphic via the lattice isomorphism $F([\A,\iota])=(\A\oplus \fC,\iota\oplus \id_\fC)$ because $\cstarlattice(\fC)$ is a single point.
\end{proof}
  However, $F$ may not implement a lattice $\ast$-isomorphism. For instance, $\T_2$, and $\T_2 \oplus \K$, where $\T_2$ is the upper triangular $2\times 2$ matrices and $\K$ is the compact operators on a separable Hilbert space, are lattice isomorphic but not lattice $*$-isomorphic. In particular, $M_2$ is a C$^*$-cover of $\T_2$ but clearly every C$^*$-cover of $\T_2\oplus \K$ must be infinite-dimensional.

Theorem \ref{thm:direct_sum_covers} fails if both $\A$ and $\B$ are not approximately unital. The next two examples illustrate this fact.
\begin{example}\label{eg:direct_sum_nonunital}
Let
\[
\A =
\text{span}(E_{12}) = 
\left\{\begin{pmatrix}
    0 & a \\
    0 & 0
\end{pmatrix}\;\middle|\; a\in \bbC\right\} \subset M_2.
\]
The homomorphism $\pi:\A\oplus \A \to M_3$ given by
\[
\pi\left(\begin{pmatrix}
    0 & a \\
    0 & 0
\end{pmatrix},\begin{pmatrix}
    0 & b \\
    0 & 0
\end{pmatrix}\right) =
\frac{1}{2}\begin{pmatrix}
    0 & a & b \\
    0 & 0 & 0 \\
    0 & 0 & 0
\end{pmatrix}
\]
is completely contractive, by the triangle inequality. Here, $C^\ast(\pi(\A\oplus \A))=M_3$ does not split as a direct sum of two proper C*-subalgebras. Moreover, if $\id_{\A\oplus \A}:\A\oplus \A\to M_2\oplus M_2$ is the identity representation, then 
\[
(M_2\oplus M_2\oplus M_3,\id_{\A\oplus \A}\oplus \pi)
\]
is a C*-cover which does not split as a direct sum of two C*-covers of $\A$.
\end{example}

\begin{example}
Let $z\in A(\bbD)$ denote the identity function on $\bbD$, the standard generator, and consider the nonunital operator subalgebra
\[
\B=\langle z\rangle =\cspan\{z,z^2,z^3,\dots\} =
\{f\in A(\bbD) \mid f(0)=0\},
\]
which is the universal \emph{nonunital} operator algebra generated by a single contraction. If $a,b\in B(H)$ are any contractions which satisfy $ab=ba=0$, then there is a completely contractive homomorphism $\pi:\B\oplus \B\to B(H)$ which is determined by $\pi(z,0)=a$ and $\pi(0,z)=b$. For instance, if we choose $a$ and $b$ such that $a^\ast b\ne 0$ or $ba^\ast\ne 0$, then $C^\ast(\pi(\B\oplus 0))C^\ast(\pi(0\oplus \B))\ne \{0\}$, and so $C^\ast(\pi(\B\oplus \B))$ does not split as an internal direct sum of $C^\ast(\pi(\B,0))$ and $C^\ast(\pi(0,\B))$. As in Example \ref{eg:direct_sum_nonunital}, direct summing with the identity representation yields many C*-covers for $\B\oplus \B$ which do not arise as direct sums.
\end{example}

\section{A one point lattice of C*-covers?}

\begin{question}\label{ques:one_point_lattice}
If $\A$ is an operator algebra which is lattice isomorphic to a C$^*$-algebra, is $\A$ itself a C$^*$-algebra? (That is, can the lattice of a properly non-selfadjoint operator algebra be a single point?)
\end{question}

For operator systems, it was shown in \cite{KirWas} that there are examples where the minimal and maximal C$^*$-covers coincide. However, their proof does not generalize to the operator case since it relies on forgetting the multiplicative structure of C$^*$-algebra. In particular, an essential difference is that the maximal C*-algebra of a C*-algebra $\fA$ \emph{considered only as an operator system} is much larger than $\fA$, because its operator system representations (i.e. ucp maps) need not be multiplicative on $\fA$. Considered as an operator \emph{algebra}, multiplicativity is forced and so $C^\ast_{\max}(\fA)=\fA$ in the setting of operator algebras. We do not know whether the lattice of C$^*$-covers can be a single point for a non-selfadjoint operator algebra but we do have some partial results in this direction.

The proof of the following is straightforward.

\begin{proposition}\label{prop:one_point_equivalence}
Let $\A$ be an operator algebra. The following are equivalent.
\begin{itemize}
\item [(1)] $\A$ is lattice isomorphic to a C*-algebra.
\item [(2)] The C*-cover morphism
\[
C^\ast_{\max}(\A)\to C^\ast_{e}(\A)
\]
is a $\ast$-isomorphism.
\item [(3)] Whenever $\iota:A\to B(H)$ is a completely isometric homomorphism, the induced $\ast$-homomorphism $\tilde{\iota}:C^\ast_{\max}(\A)\to B(H)$ with $\tilde{\iota}\vert_\A=\iota$ is injective.
\end{itemize}
\end{proposition}

\begin{remark}
Proposition \ref{prop:one_point_equivalence} implies that if $\A$ has a one point lattice, and $C^\ast_{\max}(\A)=C^\ast_{e}(\A)$ is simple, then $\A$ is simple as an operator algebra. Indeed, any representation of $\A$ extends to a $\ast$-representation of $C^\ast_{\max}(\A)$, which is automatically completely isometric. So, every representation of $\A$ is completely isometric. Although, given the discussion at the end of this paper simple operator algebras, that are not C*-algebras, are unlikely to have a one point lattice.
\end{remark}

Suppose $\A$ is an operator algebra and $\rho:\A \rightarrow B(H)$ is a completely contractive representation. Recall, that a {\em dilation} $\sigma: \A \rightarrow B(K)$ of $\rho$ is a completely contractive representation such that $H \subset K$ is a semi-invariant subspace, meaning that $\sigma$ has the following form
\[
\sigma = \left[\begin{matrix} * & 0 & 0 \\ * & \rho & 0 \\ * & * & * \end{matrix} \right].
\]
The dilation $\sigma$ is an \emph{extension} if $H$ is invariant for $\sigma(\A)$, and a \emph{co-extension} if $H$ is co-invariant for $\sigma(\A)$. Then $\rho$ is called {\em maximal} if the only dilations possible are direct sums, meaning $H$ will always be a reducing subspace of $\sigma$. Maximal dilations were introduced by Muhly and Solel \cite{MS} and then shown to exist by a direct argument by Dritschel and McCullough \cite{DriMcc}. Every maximal representation extends to a $*$-homomorphism of the C$^*$-envelope and in particular, if $\rho$ is a completely isometric maximal representation then $C^*_e(\A) \simeq C^*(\rho(\A))$. For a thorough treatment of the dilation theory of representations see \cite{DKDoc}, and for the non-unital case see \cite{DOS}, which carefully develops the theory of the unique extension property (UEP) and maximal representations.

The next result is somewhat opposite to the previous discussion of maximality.

\begin{proposition} \label{prop:non_maximal_representation}
If $\A$ is a non-selfadjoint operator algebra, then it has a non-maximal representation. 
\end{proposition}
\begin{proof}
Assume first that $\A$ is unital. Throughout this proof, we will consider $\A$ as a subalgebra of its C*-envelope $C^*_e(\A)$.
By hypothesis, there is an element $a\in \A$ such that $a^*\notin \A$.
The Hahn-Banach Separation Theorem gives that there is a linear functional $\rho$ on $C^*_e(\A)$ such that $\|\rho\|=1$, $\rho(a^*) \neq 0$ and $\rho(\A)=0$.

Now, by \cite[Theorem 8.4]{Paulsen} there exists a $*$-homomorphism $\pi: C^*_e(\A) \rightarrow B(K)$ and vectors $\xi,\eta\in K$ such that $\|\xi\| = \|\eta\|=1$ and 
\[
\rho(c) \ = \ \langle \pi(c)\xi,\eta\rangle.
\]
Let $H = \overline{\pi(\A)\xi}$ which is a closed subspace of $K$. Since $0 = \rho(\A)$ and because $\A$ is unital, $\eta \in H^\perp, \xi\in H$ and $H \neq K$.
Thus, $\sigma = P_H \pi(\cdot)|_H$ is a completely contractive representation of $\A$, by compression to an invariant subspace. However, 
\[
0\neq \rho(a^*) = \langle \pi(a^*)\xi,\eta\rangle = \langle \xi, \pi(a)\eta\rangle
\]
which implies that $P_H\pi(\A)|_{H^\perp} \neq 0$. Hence, $\sigma$ is a representation that is not extremal by extensions and so not a maximal representation.

Now if $\A$ is nonunital, then do the above argument for its unitization, $\A^1$, choosing $a\in \A$ such that $a^*\notin \A$ (this doesn't actually matter but it is easier to see the non-maximality below). This still gives that $\sigma$ is a completely contractive representation of $A$ by restriction from the unitization. Moreover, $\sigma$ is still not maximal.
\end{proof}

Recall that an operator algebra $\A$ is called {\em hyperrigid} if whenever $\pi:C^*_e(\A) \rightarrow B(H)$ is a $*$-homomorphism then $\pi$ is the only possible way to extend $\pi|_\A$ to a completely positive, completely contractive map on $C^*_e(\A)$. Dirichlet operator algebras, $\overline{\A + \A^*} = C^*_e(\A)$, are hyperrigid and the second author and Katsoulis \cite{KatRamHyper} have developed conditions for the hyperrigidity of tensor algebras of C$^*$-correspondences, providing many non-Dirichlet hyperrigid examples.

\begin{proposition}
If $\A$ is a hyperrigid non-selfadjoint operator algebra, then $\cstarlattice(\A)$ is not just a single point.
\end{proposition}
\begin{proof}
By Proposition \ref{prop:non_maximal_representation}, there exists a non-maximal representation $\rho$ of $\A$ on $B(H)$. This dilates to a maximal representation $\pi$ which extends to a $*$-homomorphism of $C^*_e(\A)$. 
This implies that $P_H \pi(\cdot)|_H$ is a completely positive, completely contractive map of $C^*_e(\A)$ that is not a $*$-homomorphism, or else the space $H$ would be a reducing subspace for $\rho$ and $\rho$ would be maximal.
Thus, $\rho$ cannot extend to a $*$-homomorphism since $\A$ is hyperrigid. Therefore, $C^*_e(\A)$ does not enjoy the universal property of $\Cmax(\A)$.
\end{proof}

This approach does not work in general since there are examples of non-maximal representations that extend to $*$-homomorphisms of the C$^*$-envelope.

\begin{example}
Consider the semicrossed product operator algebra $C([0,1]) \rtimes_\alpha \mathbb Z^+$ where $\alpha(x) = 0$ for all $x\in[0,1]$.

By \cite[Theorem 3.3]{KatRamHyper} since $\alpha([0,1]) = \{0\}$ is not contained in the closure of its interior, then $C([0,1]) \rtimes_\alpha \mathbb Z^+$ is not hyperrigid. Moreover, by \cite[Section 3]{KatRamHyper} we know that Katsura's ideal is obtained from the dynamics
\[J_X = C_0([0,1]\setminus \overline{([0,1] \setminus \alpha([0,1]))^\circ}) = C_0([0,1]\setminus \overline{(0,1]}) = \{0\}
\]
and so $C^*_e(C([0,1]) \rtimes_\alpha \mathbb Z^+)$, the Cuntz-Pimnser algebra, is $*$-isomorphic to the Toeplitz-Cuntz-Pimsner algebra (sometimes just called the Toeplitz algebra). This is the C$^*$-algebra generated by all orbit representations
\[
\sigma_x(f) = \left[\begin{matrix} f(x) \\ &f(0) \\ &&f(0) \\ &&&\ddots\end{matrix}\right] \quad \textrm{and}  \quad S =
\left[\begin{matrix} 0 \\ 1&0 \\ &1&0 \\ &&\ddots&\ddots \end{matrix}\right]
\]
Hence,
\[
C^*_e(C([0,1]) \rtimes_\alpha \mathbb Z^+)\ \  \simeq \ \ 
C^*\left(\left(\bigoplus_{x\in [0,1]} \sigma_x\right)\Big(C([0,1])\Big), \left(\bigoplus_{x\in [0,1]} S\right)  \right)
\]

Now, compression to the $0$-orbit, given by $C^*(\sigma_0(C([0,1])), S) \simeq \T$ (where $\T$ is the classical Toeplitz algebra),
is a completely contractive representation of $C([0,1]) \rtimes_\alpha \mathbb Z^+$, onto $A(\mathbb D)$, and a $*$-homomorphism of its C$^*$-envelope.
But notice that this representation on the semicrossed product algebra non-trivially dilates to
\[
f \mapsto\left[\begin{matrix} \ddots \\ &f(0) \\ &&f(0) \\ &&&\ddots\end{matrix}\right] \quad \textrm{and} \quad S \mapsto 
\left[\begin{matrix} \ddots \\ \ddots&0 \\ &1&0 \\ &&\ddots&\ddots \end{matrix}\right]
\]
which also extends to a $*$-homomorphism of the C$^*$-envelope, since this is just $\T \rightarrow C(\mathbb T)$. Therefore, a non-maximal representation of a non-hyperrigid operator algebra can still possibly extend to a $*$-homomorphism of the C$^*$-envelope.
\end{example}

Besides hyperrigidity, we also have the same result as the previous proposition for operator algebras with finite-dimensional C$^*$-envelopes.


\begin{proposition}
If $\A$ is a non-selfadjoint operator algebra whose C$^*$-envelope is finite-dimensional, then $\cstarlattice(\A)$ is not just a single point.
\end{proposition}
\begin{proof}
Suppose $C^*_e(\A) \simeq M_{n_1} \oplus \cdots \oplus M_{n_k}$. This implies that $\A = \A_1\oplus \cdots \oplus \A_k$ with $C^*(\A_i) = M_{n_i}$. 
Since $\A$ is non-selfadjoint then without loss of generality we can assume $\A_1$ is non-selfadjoint.

By Burnside's Theorem, there is a proper invariant subspace $K\subseteq\bbC^{n_1}$ for $\A_1$. Then $a\mapsto a\vert_K$ defines a completely contractive representation $\rho$ of $\A_1$ that does not extend to a $\ast$-homomorphism of $M_{n_1}$. So, $q\oplus\id\oplus \cdots\oplus \id$ is a completely contractive representation of $A$ that does not extend to the C$^*$-envelope. Therefore, the C$^*$-envelope does not have the universal property of the maximal C$^*$-cover and the lattice $\cstarlattice(\A)$ is not a single point.
\end{proof}



The situation is much clearer if we instead ask only for C$^*$-cover equivalence to a one-point lattice.

\begin{proposition}\label{prop:addallthecovers}
Let $\A$ be an operator algebra. Define 
\[
\fC_\A \ := \ \bigoplus_{[\fC,\iota] \in \cstarlattice(\A)}\bigoplus_{n=1}^\infty \ \fC.
\]
Then every C$^*$-algebra that is a C$^*$-cover of $\A\oplus \fC_\A$ is $*$-isomorphic to $\fC_\A$.  Moreover, every operator algebra is lattice isomorphic to an operator algebra which is C$^*$-cover equivalent to a C$^*$-algebra.
\end{proposition}
\begin{proof}
It should be first noted that $\fC_\A$ is properly defined since we know that $\cstarlattice(\A)$ is a set. 
By Theorem \ref{thm:direct_sum_covers} every C$^*$-cover $(\fD, \tilde\iota)$ of $\A\oplus \fC_\A$ is isomorphic to $(\fC\oplus \fC_\A, \iota\oplus\id_{\fC_\A})$ where $(\fC,\iota)$ is a C$^*$-cover of $\A$. Therefore, $\fD \simeq \fC\oplus \fC_\A \simeq \fC_\A$.
\end{proof}

\begin{corollary}
If $\A$ and $\B$ be operator algebras, then $\A$ is lattice isomorphic to $\A\oplus \fC_\A \oplus \fC_\B$ which is C$^*$-cover equivalent to $\B\oplus \fC_\A \oplus \fC_\B$ which is lattice isomorphic to $\B$.
\end{corollary}
\begin{proof}
This follows immediately from the previous proposition and  Corollary \ref{cor:addacstaralgebra}.
\end{proof}

\begin{corollary}\label{cor:makestariso}
Let $\A$ and $\B$ be operator algebras. If $\A$ is lattice isomorphic to $\B$ then $\A\oplus \fC_\A \oplus \fC_\B$ is lattice $*$-isomorphic to $\B\oplus \fC_\A \oplus \fC_\B$.
\end{corollary}

It is now easy to see that C$^*$-cover equivalence does not imply lattice isomorphism, and so they are unrelated. For instance, $\T_2 \oplus \fC_{\T_2}$ and $\fC_{\T_2}$ are C$^*$-cover equivalent but not lattice isomorphic.

\begin{theorem}\label{thm:ncstarcovers}
For any $n\in\mathbb N$ there is an operator algebra which generates exactly $n$ distinct C$^*$-algebras up to $*$-isomorphism.
\end{theorem}
\begin{proof}
The case $n=1$ is dispensed with by a C$^*$-algebra (or the previous proposition).

For $n = 2$, consider the operator algebra $\A = \left\{\left[\begin{matrix} 0&b \\ 0&0\end{matrix}\right] : b\in \mathbb C\right\}$ and its unitization $\A^1$. Loring in \cite{Loring} (cf. \cite[Example 2.4]{Blech}) proved that 
\[
\Cmax(\A) \simeq M_2\otimes C((0,1])
\]
and so the unitization satisfies
\[
\Cmax(\A^1) \simeq \{f\in M_2\otimes C([0,1]) : f(0) \in \mathbb CI\} \simeq \Cmax(\A) \oplus \mathbb C.
\]
Additionally, $M_2$ is a C$^*$-cover of both $\A^1$ and $\A$. Then $\C_\A$ is a C$^*$-cover of $\A^1\oplus \C_\A$.

$\A$ has no nonzero multiplicative linear functionals since it is made up of nilpotents. 
Thus, there is no C$^*$-cover that is $*$-isomorphic to a C$^*$-algebra that has $\mathbb C$ as a direct summand. Hence, using the C$^*$-algebra construction defined in the previous proposition, $\A^1 \oplus \fC_\A$ generates exactly two non-$*$-isomorphic C$^*$-algebras: $\fC_\A$ and $\fC_\A \oplus \mathbb C$.

For $n\geq 3$, by Theorem \ref{thm:direct_sum_covers} we know that $\bigoplus_{k=1}^{n-1} (\A^1 \oplus \fC_\A)$ generates $n$ non-$*$-isomorphic C$^*$-algebras
\[
\fC_\A, \ \  \fC_\A \oplus \mathbb C,\ \ \dots, \ \ \fC_\A \oplus \mathbb C^{n-1}
\]
since $\fC_\A \oplus \dots \oplus \fC_\A \simeq \fC_\A$.
\end{proof}

Buried in the proof of this theorem is the fact, developed by Thompson \cite[Theorem 3.1 and Example 3]{Thompson}, that the topology of C$^*$-covers need not be Hausdorff.

\section{Tensor Products}

Tensor products of operator algebras have played a significant role in C$^*$-algebras and von Neumann algebras. This is largely because many aspects are well-behaved and have lead to a rich theory while on the other hand there are many significant open questions that drive to the very core of the area. It should come as no surprise then that in the non-selfadjoint setting much can still be proven while many aspects remain mysterious. Excellent sources for further reading about tensor products are Brown and Ozawa \cite[Chapter 3]{BrownOzawa} in the selfadjoint setting and Blecher and Le Merdy \cite[Chapter 6]{Blecher} in the non-selfadjoint setting.

As is customary, $\odot$ will always denote the algebraic tensor product and $\otimes_{\max}$ and $\otimes_{\min}$ will denote the maximal and minimal operator algebra or C*-tensor products. In the case where $\A$ is an approximately unital operator algebra and $\fB$ is a nuclear C*-algebra, then $\A\otimes_{\max} \fB=\A\otimes_{\min} \fB$ \cite[Proposition 6.1.14]{Blecher}, and then one writes $\A\otimes \fB$.  

If $\A$ and $\B$ are approximately unital operator algebras, every completely contractive homomorphism
\[
\pi:\A\otimes_{\max}\B \to B(H)
\]
is of the form $\pi=\pi_\A\cdot \pi_\B$, where $\pi_\A:\A\to B(H)$ and $\pi_\B:\B\to B(H)$ are completely contractive homomorphisms with commuting ranges and
\[
(\pi_\A\cdot \pi_\B)(a\otimes b)=\pi_\A(a)\pi_\B(b)
\]
on pure tensors $a\otimes b\in \A\odot \B$ \cite[Corollary 6.1.7]{Blecher}.

We will need the following result of Murray and von Neumann.

\begin{lemma}\cite[Proposition IV.4.20]{Takesaki}\label{lem:factor_tensor}
If $\M\subseteq B(H)$ is a factor, then the multiplication map
\begin{align*}
\M'\odot \M &\to B(H) \\
a\otimes b&\mapsto ab
\end{align*}
is injective.
\end{lemma}

This has the following C*-algebraic version, which is known but not written in this explicit format. We include our proof for completeness, but note that there are other ways to achieve this, e.g. pure state excision.

\begin{proposition}\label{prop:simple_tensor_injective}
If $\fC,\fB\subseteq B(H)$ are commuting C*-algebras, and $\fB$ is simple, then the multiplication map
\begin{align*}
\mu:\fC\odot \fB &\to B(H) \\
c\otimes b &\mapsto cb
\end{align*}
is injective. In particular, there is an embedding $\fB'\odot \fB\hookrightarrow B(H)$.
\end{proposition}

\begin{proof}
Without loss of generality, by passing to $\fB+\bbC I$ and $\fC+\bbC I$, we will assume that $\fC$ and $\fB$ are unital. Note that while $\fB+\bbC I$ may no longer be simple, we will only use the fact that a unital $*$-homomorphism on $\fB+\bbC I$ that is nonzero on $\fB$ is automatically injective, which is true even when $\fB$ is nonunital and so a proper ideal in $\fB+\bbC I$.

Since $\fC$ and $\fB$ commute, they generate the C*-algebra
\[
C^\ast(\fC\cup \fB) = \overline{\fC\fB}\subseteq B(H).
\]
Let $\pi:\overline{\fC\fB}\to B(H_\pi)$ be any irreducible representation. Then, $\pi(\fC)$ and $\pi(\fB)$ are commuting C*-algebras with
\[
\pi(\fC)'\cap \pi(\fB)' = \pi(\overline{\fC\fB})' = \bbC I.
\]
It follows that both $\pi(\fC)''$ and $\pi(\fB)''$ are factors \cite[Proposition IV.4.21]{Takesaki}. Indeed, $\pi(\fC)\subseteq \pi(\fB)'$, and so $\pi(\fB)''\subseteq \pi(\fC)'$, implying
\[
\pi(\fB)''\cap \pi(\fB)' \subseteq
\pi(\fC)'\cap\pi(\fB)' =
\bbC I.
\]
Similarly, $\pi(\fC)''\cap \pi(\fC)'=\bbC I$. Therefore, Lemma \ref{lem:factor_tensor} applies with $\M=\pi(\fB)''$, and so the multiplication map
\[
\pi(\fC)\odot \pi(\fB) \to
\pi(\fC)''\odot \pi(\fB)''\subseteq
\pi(\fB)'\odot \pi(\fB)'' =
\M'\odot \M\to B(H_\pi)
\]
is injective. Because $\fB$ is simple, $\pi\vert_\fB$ is injective, and so
\[
\id_{\pi(\fC)}\otimes \pi\vert_\fB:\pi(\fC)\odot \fB\to \pi(\fC)\odot \pi(\fB)
\]
is injective. So, the diagram of $\ast$-homomorphisms
\[\begin{tikzcd}
    \fC\odot \fB \arrow[d,swap,"\mu"] \arrow[r,"\pi\vert_\fC\otimes \id_\fB"] & \pi(\fC)\odot \fB \arrow[d,hook,"\mu_\pi"] \\
    \overline{\fC\fB} \arrow[r,"\pi"] & B(H_\pi)
\end{tikzcd}\]
commutes, where
\[
\mu_\pi(\pi(c)\otimes b) =\pi(c)\pi(b)
\]
is an injective map.



Let $\Lambda$ be a set of irreducible representations $\pi:\overline{\fC\fB}\to B(H_\pi)$ that separates points of $\overline{\fC\fB}$. So, 
\[
\bigoplus_{\pi\in \Lambda} \pi\vert_\fC:\fC\to \prod_{\pi \in \Lambda}\pi(\fC)\subseteq \fB\left(\bigoplus_{\pi\in \Lambda}H_\pi\right)
\]
is an injective linear map. Since $\odot$ is injective,
\[
\left(\bigoplus_{\pi\in \Lambda} \pi\vert_\fC\right)\otimes\id_\fB: \fC\odot \fB\to \left(\prod_{\pi\in \Lambda} \pi\vert_\fC(\fC)\right)\odot \fB
\]
is also injective. Then, we have a commuting diagram
\[\begin{tikzcd}
    \fC\odot \fB \arrow[rr,"(\oplus_\pi \pi\vert_\fC)\otimes\id_\fB",hook] \arrow[d,"\mu",swap] & & (\prod_\pi \pi\vert_\fC(\fC))\odot \fB \arrow[r,hook] & \prod_\pi (\pi\vert_\fC(\fC)\odot \fB) \arrow[d,"\oplus_\pi \mu_\pi",hook] \\
    \overline{\fC\fB} \arrow[rrr,"\oplus_\pi \pi",hook] & & & \fB(\oplus_\pi H_\pi).
\end{tikzcd}\]
Here, the natural linear map
\begin{align*}
\left(\prod_\pi \pi\vert_\fC(\fC)\right)\odot \fB &\to \prod_\pi \Big(\pi\vert_\fC(\fC)\odot \fB\Big)\\
(a_\pi)_\pi \otimes b &\mapsto (a_\pi\otimes b)_\pi
\end{align*}
is injective. Since all other maps in this commuting diagram are injective, it follows that $\mu$ is injective.
\end{proof}

Now we turn to a technical result showing that tensor products of C$^*$-covers respect the lattice structure.

\begin{lemma}\label{lem:tensor_morphism}
Let $\A$ and $\B$ be operator algebras. Let $(\fC_1,\iota_1)$ and $(\fC_2,\iota_2)$ be C*-covers for $\A$, and let $(\fD_1,\eta_1)$ and $(\fD_2,\eta_2)$ be C*-covers for $\B$. Let $\alpha$ denote either ``$\max$" or ``$\min$". A $\ast$-homomorphism $\pi$ makes the diagram
\[\begin{tikzcd}
    \fC_1\otimes_\alpha \fD_1\arrow[rr,"\pi"] & & \fC_2\otimes_\alpha \fD_2\\
    \A\otimes_\alpha \B \arrow[u,"\iota_1\otimes\eta_1"] \arrow[urr,swap,"\iota_2\otimes\eta_2"]
\end{tikzcd}\]
commute if and only if it has the form $\pi_\A\otimes \pi_\B$ where
\[
\pi_\A:(\fC_1,\iota_1)\to (\fC_2,\iota_2) \quad\text{and}\quad
\pi_\B:(\fD_1,\eta_1)\to (\fD_2,\eta_2)
\]
are morphisms of C*-covers. Consequently, if $(\fC_1\otimes_\alpha \fD_1,\iota_1\otimes\eta_1)$ and $(\fC_2\otimes_\alpha \fD_2,\iota_2\otimes \eta_2)$ are C*-covers of $\A\otimes_\alpha \B$, then
\[
(\fC_1\otimes_\alpha \fD_1,\iota_1\otimes\eta_1) \succeq (\fC_2\otimes_\alpha \fD_2,\iota_2\otimes \eta_2)\iff
\begin{cases}
    (\fC_1,\iota_1)\succeq (\fC_2,\iota_2) & \text{and}\\
    (\fD_1,\eta_1)\succeq (\fD_2,\eta_2)
\end{cases}
\]
\end{lemma}

\begin{proof}
As in \cite[Section 6.1.6]{Blecher}, upon passing to the unitizations $\A^1$ and $\B^1$, we have
\[
\A\otimes_\alpha \B \subseteq \A^1\otimes_\alpha \B^1
\]
completely isometrically--and similarly for the C*-tensor products involved. 


The $\ast$-homomorphism $\pi:\C_1\otimes_\alpha \D_1\to \C_2\otimes_\alpha \D_2$ is of the form $\pi_\C\cdot \pi_\D$ for $\ast$-homomorphisms $\pi_\C:\C_1\to \C_2\otimes \D_2$ and $\pi_\D:\D_1\to \C_1\otimes \D_2$ \cite[Theorem 3.2.6]{BrownOzawa}. So, $\pi$ extends all the way to a unital $\ast$-homomorphism
\[
\tilde{\pi}=\pi_\C^1\cdot \pi_\D^1:\C_1^1\otimes_\alpha \D_1^1\to \C_2^1\otimes_\alpha \D_2^1
\]
And, $\tilde{\pi}$ splits as a tensor product of $\ast$-homomorphisms if and only if $\pi$ does. Therefore, by unitizing the algebras $\A,\B,\C_i,\D_i$, and considering the extended maps $\iota_i^1\otimes \eta_i^1$ and $\tilde{\pi}$, we can assume without loss of generality that $\A$ and $\B$, their C*-covers, and all homomorphisms involved are unital.

If $\pi$ is of the form $\pi_\A\otimes\pi_\B$, it follows readily that the diagram commutes. Conversely, suppose $\pi$ is any $\ast$-homomorphism with $\pi(\iota_1\otimes\eta_1)=\iota_2\otimes\eta_2$. Then, $\pi$ maps $\iota_1(\A)\otimes_\alpha \bbC 1_{\fD_1}$ into the subalgebra $\iota_2(\A)\otimes_\alpha\bbC 1_{\fD_2}$, and since $C^\ast(\iota_1(\A)\otimes_\alpha \bbC 1_{\fD_1})=\fC_1\otimes_\alpha\bbC 1_{\fD_1}$, it follows that $\pi$ maps $\fC_1\otimes_\alpha\bbC 1_{\fD_1}$ into $\fC_2\otimes_\alpha\bbC 1_{\fD_2}$. Symmetrically, $\pi$ maps $\bbC 1_{\fC_1}\otimes_\alpha \fD_1$ into $\bbC 1_{\fC_2}\otimes_\alpha \fD_2$. From this, it must be the case that $\pi$ has the form $\pi_\A\otimes\pi_\B$, where $\pi_\A$ and $\pi_\B$ are uniquely defined by
\[
\pi_\A(c)\otimes 1_{\fD_1}=\pi(c\otimes 1_{\fD_1}) \quad\text{and}\quad
1_{\fC_1}\otimes \pi_\B(d) = \pi(1_{\fC_1}\otimes d)
\]
for $c\in \fC_1$ and $d\in \fD_1$.
\end{proof}



If $(L,\le)$ is a lattice, we say a subset $S\subseteq L$ is \textbf{upward-closed} if for all $a \in L$, whenever there exists $b\in S$ with $a\ge b$, then $a\in S$.

\begin{theorem}\label{thm:max_tensor_lattice}
Let $\A$ be an operator algebra, and let $\fB$ be a C*-algebra. Define
\begin{align*}
\cstarlattice_{\fB,\max}(\A):=\{[\fC,\iota]\mid \iota\otimes \id_\fB:\A\otimes_{\max} \fB\to \fC\otimes_{\max} \fB \text{ is completely isometric} \}.
\end{align*}
\begin{itemize}
\item [(1)] $\cstarlattice_{\fB,\max}(\A)$ is a nonempty upward-closed complete sublattice of $\cstarlattice(\A)$.
\item [(2)] There is an order injection
\begin{align*}
    \Phi:\cstarlattice_{\fB,\max}(\A) &\hookrightarrow \cstarlattice(\A\otimes_{\max}\fB) \\
    [\fC,\iota] &\mapsto [\fC\otimes_{\max} \fB,\iota\otimes \id_\fB].
\end{align*}
\item [(3)] If $\fB$ is nuclear, then
\[
\cstarlattice_{\fB,\max}(\A) = \cstarlattice(\A).
\]
\item [(4)] If $\fB$ is simple, then every C*-cover for $\A\otimes_{\max} \fB$ is isomorphic to one of the form
\[
(\fC\otimes_\alpha \fB,\iota\otimes \id_\fB),
\]
where $(\fC,\iota)$ is a C*-cover for $\A$ satisfying $[\fC,\iota]\in \cstarlattice_{\fB,\max}(\A)$, and $\otimes_\alpha$ denotes a completion of $\fC\odot \fB$ with respect to some C*-norm $\|\cdot\|_\alpha$.
\item [(5)] If $\fB$ is nuclear and simple, then $\Phi$ is a lattice isomorphism
\[
\cstarlattice(\A) \cong \cstarlattice(\A\otimes \fB).
\]
In particular, every C*-cover of $\A\otimes \fB$ has the form $\fC\otimes \fB$ for a C*-cover $\fC$ of $\A$.
\end{itemize}
\end{theorem}

\begin{proof}
(1) In \cite[Section 6.1.9]{Blecher}, Blecher and Le Merdy show that
\[
\A\otimes_{\max} \fB\to C^\ast_{\max}(\A) \otimes_{\max} \fB
\]
is completely isometric using only the universal property of $C^\ast_{\max}(\A)$. So, $\cstarlattice_{\fB,\max}(\A)$ contains $C^\ast_{\max}(\A)$ and is always nonempty. To show that it is upward closed, suppose $[\fC,\iota]\in \cstarlattice(\A)$ and $[\fD,\eta]\in \cstarlattice_{\fB,\max}(\A)$ with $[\fC,\iota]\succeq [\fD,\iota]$. Then there is a morphism $\pi:(\fC,\iota)\to (\fD,\eta)$ of C*-covers. Taking a maximal tensor product with $\fB$ shows that the diagram
\[\begin{tikzcd}
    \fC\otimes_{\max}\fB \arrow[r,"\pi\otimes \id_\fB"] & \fD\otimes_{\max} \fB\\
    \A\otimes_{\max}\fB \arrow[u,"\iota\otimes\id_\fB"] \arrow[ur,swap,"\eta\otimes\id_\fB"]
\end{tikzcd}\]
commutes. Since $\eta\otimes\id_\fB$ is completely isometric and $\pi\otimes \id_\fB$ is completely contractive, this implies that $\iota\otimes\id_\fB$ is completely isometric. So, $[\fC,\iota]\in \cstarlattice_{\fB,\max}(\A)$.

Being upward closed, $\cstarlattice_{\fB,\max}(\A)$ is closed under arbitrary joins. To show that $\cstarlattice_{\fB,\max}(\A)$ is closed under arbitrary meets, by the description of the meet in Proposition \ref{prop:cstarlattice}, it suffices to assume that $(\C,\iota)$ is a C*-cover for $\A$, and $\I_\lambda$, $\lambda\in \Lambda$, are boundary ideals with quotient maps $q_\lambda:\C\to \C/\I_\lambda$ such that $[\C/\I_\lambda,q_\lambda\iota]$ is in $\cstarlattice_{\B,\max}(\A)$ for each $\lambda\in \Lambda$. Then, for $\I=\overline{\sum_{\lambda \in \Lambda}\I_\lambda}$ with quotient map $q$, we will show that $[\C/\I,q\iota]$ is also in $\cstarlattice_{\B,\max}(\A)$. Since the maximal tensor product $\otimes_{\max}$ of C*-algebras is exact \cite[Proposition 3.7.1]{BrownOzawa}, for any closed ideal $\J\triangleleft \C$, the natural map $\J\otimes_{\max}\B\to \C\otimes_{\max} \B$ is injective and so we identify $\J\otimes_{\max} \B\subseteq \C\otimes_{\max} \B$ as a literal subset which satisfies
\[
\frac{\C\otimes_{\max}\B}{\J\otimes_{\max}\B} \cong
\frac{\C}{\J}\otimes_{\max}\B
\]
via the natural map. Then, when identified as subsets of $\C\otimes_{\max}\B$, we have
\[
\I\otimes_{\max} \B =
\left(\overline{\sum_{\lambda\in \Lambda} \I_\lambda}\right)\otimes_{\max} \B =
\overline{\sum_{\lambda\in \Lambda}
    \left(\I_\lambda\otimes_{\max} \B\right)
}.
\]
Since each $\I_\lambda\otimes_{\max} \B$ is a boundary ideal for $(\iota\otimes\id_\B)(\A\otimes_{\max} \B)$, and the boundary ideals form a complete sublattice of the ideal lattice in $\C\otimes_{\max} \B$, the ideal $\I\otimes_{\max}\B$ is also a boundary ideal. That is, the map
\[
q\iota:\A\otimes_{\max}\B \to
\left(\frac{\C}{\I}\right)\otimes_{\max} \B \cong
\frac{\C\otimes_{\max}\B}{\I\otimes_{\max}\B}
\]
is completely isometric, and so $[\C/\I,q\iota]$ is in $\cstarlattice_{\B,\max}(\A)$. This proves $\cstarlattice_{\B,\max}(\A)$ is closed under arbitrary meets and so is a complete lattice.


\bigskip

(2) Lemma \ref{lem:tensor_morphism} shows that $\Phi$ is a well-defined order injection.

\bigskip

(3) If $\fB$ is a nuclear C*-algebra, then $\C\otimes_{\max} \fB=\C\otimes_{\min} \fB=:\C\otimes \fB$ for every operator algebra $\C$. Since the minimal tensor product $\otimes_{\min}$ preserves complete isometry, $\iota\otimes \id_\fB:\A\otimes \fB\to \fC\otimes \fB$ is completely isometric for every C*-cover $(\fC,\iota)$ of $\A$. So,
\[
\cstarlattice_{\fB,\max}(\A)=\cstarlattice(\A).
\]

\bigskip

(4) Suppose that $\fB$ is simple. Let $(\fD,\eta)$ be a C*-cover for $\A\otimes_{\max}\fB$. Then we have $\eta=\iota\cdot \sigma$, where $\iota:\A\to \fD$ is a completely isometric homomorphism, and $\sigma:\fB\to \fD$ is an injective $\ast$-homomorphism. Set $\fC:= C^\ast(\iota(\A))$, so that $(\fC,\iota)$ is a C*-cover of $\A$. Since $\sigma(\fB)\cong \fB$ is simple and commutes with $\C$, Proposition \ref{prop:simple_tensor_injective} shows that the multiplication map
\[
\fC\odot \fB \xrightarrow[]{\id_\fC\cdot \sigma} \fD
\]
is injective. Let $\|\cdot\|_\alpha$ be the norm induced on $\fC\odot \fB$ via this injection, and then $\id_\fC\cdot  \sigma$ extends to an injective $\ast$-homomorphism on the completion $\fC\otimes_\alpha \fB$ with respect to the C*-norm $\|\cdot\|_\alpha$. Then, the diagram
\[\begin{tikzcd}
    \fC\otimes_\alpha \fB \arrow[r,"\id_\fC\cdot \sigma",hook] & \fD \\
    \A\otimes_{\max} \fB \arrow[u,"\iota\otimes\id_\fB",hook] \arrow[ur,swap,hook,"\eta"]
\end{tikzcd}\]
commutes. Since $\fD=C^\ast(\eta(\A\otimes_{\max}\fB))$, it also follows that $\id_\fC\cdot \sigma$ is surjective and so an isomorphism
\[
(\fD,\eta)\cong (\fC\otimes_{\alpha} \fB,\iota\otimes \id_\fB)
\]
of C*-covers. Finally, using the unique $\ast$-homomorphism $\fC\otimes_{\max} \fB\to \fC\otimes_\alpha \fB$, the diagram
\[\begin{tikzcd}
    \fC\otimes_{\max} \fB \arrow[r] & \fC\otimes_\alpha \fB \\
    \A\otimes_{\max} \fB\arrow[u,"\iota\otimes \id_\fB"] \arrow[ur,swap,hook,"\iota\otimes \id_\fB"]
\end{tikzcd}\]
commutes, and since the diagonal map is completely isometric, so is the vertical map. Hence, this shows $\fC\in \cstarlattice_{\fB,\max}(\A)$.

\bigskip

(5) If $\fB$ is both nuclear and simple, then statements (2) and (3) show that $\Phi$ gives an order injection $\cstarlattice(\A)\to \cstarlattice(\A\otimes \fB)$. Item (4) implies that $\Phi$ is also surjective, because nuclearity implies that $\fC\otimes_\alpha \fB=\fC\otimes_{\max} \fB$ for any C*-cover $\fC$.
\end{proof}


Using the standard assignment $M_n(\A)\cong \A\otimes M_n$ for any operator algebra, we recover the following result, which can be proved on its own using basic representation theory of $M_n$.

\begin{corollary}
If $\A$ is an operator algebra, then $\A$ and $M_n(\A)$ are lattice isomorphic via the assignment
\[
[\fC,\iota]\mapsto [M_n(\fC),\iota^{(n)}].
\]
\end{corollary}

The theory discussed above then also gives a complementary theory for the min tensor product. Note that the final sentence in the following Proposition was proved without any restriction on simplicity first in \cite[Corollary 2.7]{DOEG}, but we are able to obtain it again in the simple case as a consequence of our methods.

\begin{proposition}\label{prop:min_tensor_lattice}
Let $\A$ be an operator algebra, and let $\fB$ be a C*-algebra. There is an order injection
\begin{align*}
    \Psi:\cstarlattice(\A) &\hookrightarrow \cstarlattice(\A\otimes_{\min}\fB) \\
    [\fC,\iota] &\mapsto [\fC\otimes_{\min} \fB,\iota\otimes \id_\fB].
\end{align*}
Moreover, if $\fB$ is simple, then every C*-cover for $\A\otimes_{\min} \fB$ is isomorphic to one of the form $(\fC\otimes_\alpha \fB,\iota\otimes\id_\fB)$, where $(\fC,\iota)$ is a C*-cover for $\A$, and $\iota\otimes \id_\fB:\A\otimes_{\min}\fB\to \fC\otimes_\alpha \fB$ is completely isometric. Consequently, if $\fB$ is simple, then
\[
C^*_{e}(\A\otimes_{\min}\fB) \cong
C^*_{e}(\A)\otimes_{\min}\fB.
\]
\end{proposition}

\begin{proof}
Since the minimum tensor product $\otimes_{\min}$ is injective, it follows that if $\iota:\A\to \fC$ is completely isometric, then $\iota\otimes 1_\fB:\A\otimes_{\min}\fB\to \fC\otimes_{\min} \fB$ is completely isometric. Therefore $\Psi$ is well defined. Lemma \ref{lem:tensor_morphism} with $\alpha=\min$ shows that $\Psi$ is an order injection.

If $\fB$ is simple, then the same argument as in the proof of Theorem \ref{thm:max_tensor_lattice}.(3) with $\otimes_{\min}$ in place of $\otimes_{\max}$ shows that any C*-cover is a C*-completion of $\fC\odot \fB$, for some C*-cover $(\fC,\iota)$. (However, we cannot conclude from this that $\fC\in \cstarlattice_{\fB,\max}(\A)$.) In particular, we have
\[
C^\ast_{e}(\A\otimes_{\min} \fB)\cong
\fC\otimes_\alpha \fB
\]
for some C*-cover $(\fC,\iota)$. If $\pi:(\fC,\iota)\to (C^\ast_{e}(\A),\eta)$ is the universal homomorphism to the C*-envelope, and $q:\fC\otimes_\alpha \fB\to \fC \otimes_{\min} \fB$ is the quotient to the minimal tensor product, then the following diagram commutes:
\[\begin{tikzcd}
    \fC\otimes_\alpha \fB \arrow[r,"q"] & \fC\otimes_{\min} \fB \arrow[r,"\pi\otimes\id_\fB"] & C^\ast_{e}(\A)\otimes_{\min} \fB \\
    & \A\otimes_{\min} \fB \arrow[ul,hook,"\iota\otimes \id_\fB"] \arrow[u,hook,"\iota\otimes\id_\fB"] \arrow[ur,hook,swap,"\epsilon\otimes\id_\fB"].
\end{tikzcd}\]
Therefore $(\pi\otimes\id_\fB)q$ is a morphism of C*-covers, and therefore a $\ast$-isomorphism
\[
C^\ast_{e}(\A\otimes_{\min} \fB)=\fC\otimes_\alpha \fB\cong C^\ast_{e}(\A)\otimes_{\min} \fB. \qedhere
\]
\end{proof}


The following example shows that if $\fB$ is not simple, then Theorem \ref{thm:max_tensor_lattice}.(3) and the corresponding statement in Proposition \ref{prop:min_tensor_lattice} both may fail.

\begin{example}
If $\A$ is any operator algebra, then we have
\[
\A\otimes \bbC^2 \cong \A\oplus \A
\]
completely isometrically. Then, for any pair $(\fC,\iota)$ and $(\fD,\eta)$ of C*-covers for $\A$, $(\fC\oplus \fD,\iota\oplus \eta)$ is a C*-cover for $\A\oplus \A$. However, it is not isomorphic to one of the form $(\fB\otimes \bbC^2,\rho\otimes\id_{\bbC^2})\cong (\fB\oplus \fB,\rho\oplus \rho)$ for a C*-cover $(\fB,\rho)$ unless
\[
(\fC,\iota)\cong (\fB,\rho)\cong (\fD,\eta).
\]
Indeed, a C*-cover isomorphism
\[\begin{tikzcd}
    \fC\oplus \fD \arrow[r,"\pi"] & \fB\oplus \fB \\
    \A\oplus \A \arrow[u,"\iota\oplus \eta",hook] \arrow[ur,swap,hook,"\rho\oplus \rho"]
\end{tikzcd}\]
must be of the form $\alpha\oplus \beta$, where $\alpha:(\fC,\iota)\to (\fB,\eta)$ and $\beta:(\fD,\eta)\to (\fB,\rho)$ are isomorphisms. (This same argument appears in the proof of Theorem \ref{thm:direct_sum_covers}.)
\end{example}

If neither $\A$ nor $\B$ is a C*-algebra, the C*-covers of a tensor product $\A\otimes_\alpha \B$ may be even more ill-behaved. A fundamental obstruction is that in a (completely isometric) representation of $\A\otimes_\alpha \B$, the images of $\A$ and $\B$ commute but need not $\ast$-commute.

\begin{example}
If $\A$ and $\B$ are non-selfadjoint, then $\A\otimes_{\max} \B$ and $\A\otimes_{\min} \B$ can have C*-covers which are not tensor products of C*-covers of $\A$ and $\B$. For instance, let $\A=\B=A(\bbD)$ be the disk algebra. Ando's dilation theorem implies that
\[
A(\bbD)\otimes_{\max} A(\bbD) = A(\bbD)\otimes_{\min}A(\bbD),
\]
see \cite[Section 6.2]{Blecher}, and we will write both as simply $A(\bbD)\otimes A(\bbD)$.

Let $\iota:A(\bbD)\to C(\bbT)$ be the inclusion of $A(\bbD)$ into its C*-envelope (or, any completely isometric representation of $A(\bbD)$ gives a working example here). Since $\iota$ is completely isometric, and $\otimes_{\min}$ is injective, 
\[
\iota\otimes\iota :A(\bbD)\otimes A(\bbD)\to C(\bbT)\otimes C(\bbT)
\]
is completely isometric. Let $S\in B(\ell^2)$ be the unilateral shift, and $\pi:A(\bbD)\to B(\ell^2)$ be the completely isometric representation determined by mapping the generator $z$ in $A(\bbD)$ to $S$. Then,
\[
\pi\cdot \pi:A(\bbD)\otimes A(\bbD) \to B(\ell^2)
\]
is a completely contractive representation of $A(\bbD)\otimes A(\bbD)$. Then $\sigma=(\iota\otimes \iota)\oplus (\pi\cdot\pi)$ is a completely isometric representation. Here,
\[
\sigma(z\otimes 1) = (z\otimes 1)\oplus S \quad\text{and}\quad \sigma(1\otimes z) =(1\otimes z)\oplus S
\]
commute, but do not $\ast$-commute because $S$ is not normal. Therefore
\[
(C^\ast(\sigma(A(\bbD)\otimes A(\bbD))),\sigma)
\]
is a C*-cover which is not of the form
\[
(\fC\otimes_\alpha \fD,\alpha\otimes \beta)
\]
for any C*-tensor product $\otimes_\alpha$ and C*-covers $(\fC,\alpha),(\fD,\beta)$ for $A(\bbD)$. This also shows that the images of $z\otimes 1$ and $1\otimes z$ do not $\ast$-commute in the maximal C*-algebra of $A(\bbD)\otimes A(\bbD)$, and so the morphism of C*-covers
\[
C^\ast_{\max}(A(\bbD)\otimes A(\bbD)) \to C^\ast_{\max}(A(\bbD))\otimes_{\max} C^\ast_{\max}(A(\bbD))
\]
is not a $\ast$-isomorphism.
\end{example}

Finally, we can complete our task of showing the five equivalences of operator algebras are all different. Much of this has been proven already but it will greatly help the reader to have this all in one place.

\begin{theorem}\label{thm:allrelationsdifferent}
Completely isometrically isomorphic is strictly stronger than lattice intertwined which is strictly stronger than lattice $*$-isomorphic which is strictly stronger than either of lattice isomorphic or C$^*$-cover equivalent. Moreover, lattice isomorphic and C$^*$-cover equivalent are unrelated.
\end{theorem}
\begin{proof}
Proposition \ref{prop:intertwinedwithopposite} and Example \ref{ex:intertwinednotcii} showed that there is an operator algebra $\A$ such that $\A$ is not completely isometrically isomorphic to $\A^*$ but $\A$ and $\A^*$ are always lattice intertwined.

By Corollary \ref{cor:addacstaralgebra} we know that $\T_2$ and $\T_2\oplus \K$ are lattice isomorphic but not C$^*$-cover equivalent and so not lattice $*$-isomorphic either.

By Proposition \ref{prop:addallthecovers} we know that $\T_2\oplus \fC_{\T_2}$ and $\fC_{\T_2}$ are C$^*$-cover equivalent but not lattice isomorphic and so not lattice $*$-isomorphic either.

The only thing left to prove is to show lattice $*$-isomorphic does not imply lattice intertwined. 
To this end, Theorem \ref{thm:max_tensor_lattice} gives us that $\T_2$ and $M_2(\T_2)$ are lattice isomorphic. By Corollary \ref{cor:makestariso} $\T_2 \oplus \fC_{\T_2} \oplus \fC_{M_2(\T_2)}$ and $M_2(\T_2) \oplus \fC_{\T_2} \oplus \fC_{M_2(\T_2)}$ are lattice $*$-isomorphic. Again by \cite{Loring} and \cite[Example 2.4]{Blech} we know that 
\[
\Cmax(\T_2 \oplus \fC_{\T_2} \oplus \fC_{M_2(\T_2)}) \simeq \{f\in M_2(C([0,1])) : f(0) \ \textrm{is diagonal}\} \oplus \fC_{\T_2} \oplus \fC_{M_2(\T_2)}
\]
and 
\[
\Cmax(M_2(\T_2) \oplus \fC_{\T_2} \oplus \fC_{M_2(\T_2)}) \simeq M_2(\{f\in M_2(C([0,1])) : f(0) \ \textrm{is diagonal}\}) \oplus \fC_{\T_2} \oplus \fC_{M_2(\T_2)}.
\]
As well,
\[
C^*_e(\T_2 \oplus \fC_{\T_2} \oplus \fC_{M_2(\T_2)}) \simeq M_2 \oplus \fC_{\T_2} \oplus \fC_{M_2(\T_2)}
\]
and
\[
C^*_e(M_2(\T_2) \oplus \fC_{\T_2} \oplus \fC_{M_2(\T_2)}) \simeq M_4 \oplus \fC_{\T_2} \oplus \fC_{M_2(\T_2)}.
\]
Thus, the canonical quotient maps from the max covers to the min covers, point evaluation at 1, have kernels
\[
\{f\in M_2(C([0,1)) : f(0) \ \textrm{is diagonal}\} \ \ \textrm{and} \ \ M_2(\{f\in M_2(C([0,1])) : f(0) \ \textrm{is diagonal}\}).
\]
These ideals are not $*$-isomorphic since the latter has a $*$-homomorphism onto $M_4$, e.g. point evaluation at 1/2, while the former does not.
Therefore, $\T_2 \oplus \fC_{\T_2} \oplus \fC_{M_2(\T_2)}$ and $M_2(\T_2) \oplus \fC_{\T_2} \oplus \fC_{M_2(\T_2)}$ are lattice $*$-isomorphic but they are not lattice intertwined by Proposition \ref{prop:intertwining}.
\end{proof}

\section{Properties preserved (or not) by C*-lattice equivalences}

\subsection{Residually finite-dimensional}

An operator algebra is {\bf residually finite-dimensional (RFD)} if it is completely isometrically isomorphic to a subalgebra of a (possibly infinite) product of matrix algebras \cite{CR}. 
In other words, an operator algebra is RFD if and only if it has an RFD C$^*$-cover. Thus, RFD is preserved by C$^*$-cover equivalence, and so by lattice $*$-isomorphism, lattice intertwining, and completely isometric isomorphism, but not by lattice isomorphism (e.g. $\mathbb C$ and $\K$ are lattice isomorphic).

Given an RFD operator algebra, the subset of C$^*$-covers that are RFD is an object of some interest. It has been shown that there are examples of an RFD operator algebra where the C$^*$-envelope need not be RFD \cite[Example 3]{CR}, even when the operator algebra is finite-dimensional \cite{Hamidi}, and recently by Hartz that the maximal C$^*$-cover need not be RFD either \cite{Hartz}. However, \cite{Thompson} shows that there is always a maximal RFD C$^*$-cover since the RFD property is preserved by direct sums. See also \cite{CDO} for more on when the maximal C*-cover is RFD.

\begin{proposition}
Every C$^*$-cover of $\T_2$ is RFD.
\end{proposition}
\begin{proof}
As we have seen, 
\[
\Cmax(\T_2) \simeq \{f \in C([0,1]) \otimes M_2 : f(0) \ \textrm{is diagonal}\}.
\]
This implies that for every C$^*$-cover $(\fC,\iota)$ there exists a compact set $X\subseteq[0,1]$ such that $\fC$ is $*$-isomorphic to a subalgebra of $C(X)\otimes M_2$.
Therefore, every C$^*$-cover is RFD.
\end{proof}

We do not know if the subset of RFD C$^*$-covers forms a complete sublattice as in the previous proposition. Certainly it forms a complete upper semilattice since the join of RFD C$^*$-covers (direct sum) is RFD. However, it is not known whether there is always a minimal RFD C$^*$-cover.

In summary:
\begin{proposition}
Let $\A$ and $\B$ be operator algebras.
\begin{itemize}
\item[(i)] If $\A$ and $\B$ are C$^*$-cover equivalent, then $\A$ is RFD if and only if $\B$ is RFD. 
\item[(ii)] If $\A$ and $\B$ are lattice $*$-isomorphic, then the subsets of RFD C$^*$-covers of $\A$ and $\B$ are in bijective correspondence. 
\item[(iii)] If $\A$ and $\B$ are lattice intertwined, then the complete upper semilattices of RFD C$^*$-covers of $\A$ and $\B$ are intertwined.
\end{itemize}
\end{proposition}
\begin{proof}
For (i), if $\A$ is RFD it must have a C$^*$-cover that is RFD as well. Thus, by the C$^*$-cover equivalence with $\B$ we see that $\B$ must then have an RFD C$^*$-cover and hence is RFD. 

For (ii), if $\A$ and $\B$ are lattice $*$-isomorphic, then there is an order isomorphism $F : \cstarlattice(\A) \rightarrow \cstarlattice(\B)$ such that $F([\C, \iota]) = [\D,\eta]$ implies that $\C$ and $\D$ are $*$-isomorphic. Since the property being RFD (or not) is preserved by $\ast$-isomorphisms, the map $F$ puts the RFD C$^*$-covers of $\A$ and $\B$ in bijective correspondence.

For (iii), lattice intertwined is stronger than lattice $*$-isomorphic so such an $F$ in the proof of part (ii) exists but also intertwines the C$^*$-covers according to the lattice structure, see Definition \ref{def:sameness}.
\end{proof}

\subsection{Action-admissible}

Consider a {\bf dynamical system} $(\A, G, \alpha)$, that is, $\A$ is an operator algebra, $G$ is a locally compact group, and $\alpha$ is a strongly continuous action of $G$ on $\A$ by completely isometric automorphisms. A C$^*$-cover $(\fC,\iota)$ is called {\bf $\alpha$-admissible} if there exists a strongly continuous $\alpha_\fC : G \rightarrow \Aut(\fC)$ such that $\alpha_\fC\iota = \iota\alpha$ on $\A$. This was first defined in \cite{KatRamMem} and they showed that the minimal and maximal C$^*$-covers are always $\alpha$-admissible. 

It was shown in \cite{Hamidi} that the subset of $\alpha$-admissible C$^*$-covers is a complete sublattice. This complete sublattice need not be the whole lattice as there are examples of non-$\alpha$-admissible C$^*$-covers in \cite{KatRamHN} and \cite{Hamidi}.

On the other hand, every completely isometric automorphism of $\T_n$ is a diagonal unitary inner automorphism. Thus, every C$^*$-cover is admissible for every completely isometric automorphism of $\T_n$.

\begin{example}
Consider the algebra 
\[
\A = 
\left\{\begin{pmatrix}
    a & b & c \\
    0 & a & 0 \\
    0 & 0 & a
\end{pmatrix}\;\middle|\;
a,b,c\in \bbC\right\}
\]
from Example \ref{ex:intertwinednotcii}, which is lattice intertwined with $\A^*$ but not completely isometrically isomorphic. We know that $C^*_e(\A) = M_3$ and so every completely isometric automorphism of $\A\oplus \A$ extends to a $*$-automorphism of $M_3 \oplus M_3$. Hence, every such automorphism preserving $\A\oplus \A$ is a diagonal unitary inner automorphism in each component and a $\mathbb Z_2$ action which permutes the components.

Consider now $\A\oplus \A^*$ whose completely isometric automorphisms are now only the diagonal unitary inner automorphisms in each component, because $\A$ and $\A^*$ are not completely isometrically isomorphic. However, the $\mathbb Z_2$ action on $M_3\oplus M_3$ does not preserve $\A\oplus \A^*$. Therefore, group actions may not be preserved by lattice intertwining, in the sense that the admissible covers of a given action don't correspond under the intertwining.
\end{example}

\subsection{Simplicity}

First off, it is immediate that lattice isomorphism doesn't preserve simplicity since all C$^*$-algebras are lattice isomorphic. However, saying anything about the other equivalences is not really possible since, as far as the authors are aware, no one has a good understanding of what a simple operator algebra is. What is more, simplicity is really a completely bounded isomorphism property. 

Similarities of simple C$^*$-algebras give the first examples of simple non-selfadjoint operator algebras. To give us some idea of the form of the lattice of C$^*$-covers of a simple non-selfadjoint operator algebra we turn to the simplest simple operator algebra.

\begin{proposition}\label{prop:intertwinedidempotents}
If $E_1$ and $E_2$ are non-selfadjoint idempotents then $\mathbb CE_1$ and $\mathbb C E_2$ are lattice intertwined. 
\end{proposition}
\begin{proof}
It is probably folklore, but something similar can be found in \cite[Example 4.4]{Grill}, that 
\[
\Cmax(\mathbb C E) \ \simeq \ \left\{ f\in M_2\left(C\left(\left[0, \sqrt{\|E\|^2-1}\right]\right)\right) : f(0) \in \left[\begin{matrix} \mathbb C & 0 \\ 0 & 0\end{matrix}\right]\right\},
\]
where the embedding of $\mathbb C E$ is given by $E \mapsto \left[\begin{matrix}1 & x \\ 0 & 0\end{matrix}\right]$. 
Of course, $\left[0, \sqrt{\|E\|^2-1}\right]$ is homoemorphic to $[0,1]$ by $x\mapsto \frac{1}{\sqrt{\|E\|^2-1}}x$. This implies that 
\[
\Cmax(\mathbb C E) \ \simeq \ \left\{ f\in M_2\left(C\left(\left[0, 1\right]\right)\right) : f(0) \in \left[\begin{matrix} \mathbb C & 0 \\ 0 & 0\end{matrix}\right]\right\},
\]
where the embedding of $\mathbb C E$ is given by $E \mapsto \left[\begin{matrix}1 & \sqrt{\|E\|^2-1}x \\ 0 & 0\end{matrix}\right]$. 

Now, for a non-selfadjoint idempotent $E$ one has $C^*_e(\bbC E) \simeq M_2$ given by 
\[
E \mapsto \left[\begin{matrix} 1 & \sqrt{\|E\|^2-1} \\ 0 & 0 \end{matrix}\right].
\]
Hence, for two non-selfadjoint idempotents $E_1$ and $E_2$ we have
\[\begin{tikzcd}
    \Cmax(\bbC E_1) \arrow[d, "\rho_1"] \arrow[r,"\Phi"] & \Cmax(\bbC E_2) \arrow[d, "\rho_1"] \\
    C^*_e(\bbC E_1) \arrow[r,"\phi"] & C^*_e(\bbC E_2)
\end{tikzcd}\]
where $\rho_1$ is point evaluation at 1 and where $\Phi$ and $\phi$ are the natural $*$-isomorphisms running through the intermediate algebras $\left\{ f\in M_2\left(C\left(\left[0, 1\right]\right)\right) : f(0) \in \left[\begin{matrix} \mathbb C & 0 \\ 0 & 0\end{matrix}\right]\right\}$ and $M_2$, respectively.
Therefore, $\bbC E_1$ and $\bbC E_2$ are lattice intertwined by Proposition \ref{prop:intertwining}.
\end{proof}

So are there simple operator algebras not similar to C$^*$-algebras? We will spend the rest of this section answering in the affirmative.

Let $H=\ell^2$ be a separable Hilbert space with basis $\{e_1,e_2,\ldots,\}$. In $B(H\otimes \bbC^2)$, define
\[
\tilde{E}_{ij} =
E_{ij}\otimes \begin{pmatrix} 1 & j \\ 0 & 0 \end{pmatrix},
\]
where $E_{ij}$ are the standard matrix units in $B(H)$. We have the relation
\[
\tilde{E}_{ij}\tilde{E}_{k\ell} = \delta_{j,k}\tilde{E}_{i\ell},
\]
and so
\[
\A :=
\cspan\{\tilde{E}_{ij}\mid i,j\ge 1\}
\]
is an operator algebra.

Here is another description of $A$. If we let
\[
J_n :=
\begin{pmatrix}
    1 & -n \\
    0 & 1
\end{pmatrix}
\]
for $n\in\mathbb{R}$, then up to the standard unitary identifying $M_n(M_2)$ with $M_n\otimes M_2$, and the standard identification $K(H)=\overline{\bigcup_n M_n}$, we have
\[
\A =
\overline{\bigcup_n (J_1\oplus \cdots \oplus J_n)M_n(\bbC\oplus 0) (J_1\oplus \cdots\oplus J_n)^{-1}}\subseteq K(H)\otimes M_2 = K(H\otimes \bbC^2).
\]
That is, roughly speaking, $\A$ is built from the inductive limit $\overline{\bigcup_n M_n}$ by applying a sequence of similarities to each copy of $M_n$ whose norm tends to infinity. We show in Theorem \ref{thm:simple_not_cb_iso_to_C*} below that the resulting algebra $\A$ is not completely boundedly isomorphic to any C*-algebra.

\begin{lemma}\label{lem:A_approx_identity}
For all $a\in \A$, we have
\begin{itemize}
\item [(i)] $\tilde{E}_{ij}a=(E_{ij}\otimes 1_2)a$ for all $i,j\ge 1$,
\item [(ii)] $a\tilde{E}_{ii}=a(E_{ii}\otimes 1_2)$ for all $i\ge 1$, and
\item [(iii)] if $u_n=\sum_{k=1}^n \tilde{E}_{kk}$, then $u_nau_n\to a$ as $n\to\infty$.
\end{itemize}
\end{lemma}

\begin{proof}
Items (i) and (ii) follow by checking on the generators $\tilde{E}_{k\ell}$, $k,\ell\ge 1$. Indeed
\begin{align*}
\tilde{E}_{ij}\tilde{E}_{k\ell} &=
\delta_{j,k}\left(E_{i,\ell}\otimes \begin{pmatrix} 1 & \ell \\ 0 & 0 \end{pmatrix}\right) =
(E_{ij}\otimes 1_2)\tilde{E}_{k\ell},
\end{align*}
and
\begin{align*}
\tilde{E}_{k\ell}\tilde{E}_{ii} = 
\delta_{\ell,i}\left(E_{k,i}\otimes \begin{pmatrix} 1 & i \\ 0 & 0 \end{pmatrix}\right) =
\delta_{\ell,i}\left(E_{k,i}\otimes \begin{pmatrix} 1 & \ell \\ 0 & 0 \end{pmatrix}\right) =
\tilde{E}_{k,\ell}(E_{ii}\otimes 1_2).
\end{align*}

Let $v_n = \sum_{k=1}^n E_{kk}\in K(H)$. Then combining (i) and (ii) shows that for $a\in \A$,
\[
u_nau_n =
(v_n\otimes 1_2)a(v_n\otimes 1_2).
\]
Since $v_n$ is an approximate identity for $K(H)$, $v_n\otimes 1_2$ is an approximate identity for $K(H)\otimes M_2\supseteq \A$, and so $u_nau_n=(v_n\otimes 1_2)a(v_n\otimes 1_2)\to a$ as $n\to\infty$.
\end{proof}

\begin{proposition}
$\A$ is a simple operator algebra.
\end{proposition}

\begin{proof}
Suppose that $J\subseteq \A$ is a nonzero closed ideal, so there is some $a\ne 0$ in $J$. By Lemma \ref{lem:A_approx_identity}, we must have $u_nau_n\ne 0$ for some $n$. Since $u_n=\sum_{k=1}^n \tilde{E}_{kk}$, there are some $i,j\ge 1$ with
\[
\tilde{E}_{ii}a\tilde{E}_{jj} \ne 0.
\]
Then,
\[
\tilde{E}_{1i}\tilde{E}_{ii}a\tilde{E}_{jj}\tilde{E}_{j1}
\]
is nonzero, is in $J$, and is in $\tilde{E}_{11}\A\tilde{E}_{11}= \bbC \tilde{E}_{11}$. Therefore $\tilde{E}_{11}\in J$, and it follows that $\tilde{E}_{ij}\in J$ for all $i,j\ge 1$. Since $J$ is a closed subspace, we have $J=\A$.
\end{proof}

Let $H_0:= H\otimes \bbC e_1\subseteq H\otimes \bbC^2$. Then checking on generators shows that $H_0$ is invariant for $\A$.

\begin{proposition}\label{prop:A_restricts_to_K}
The restriction map $\pi:\A\to K(H_0)$ is an injective completely contractive homomorphism with dense range. However, $\pi$ is not onto.
\end{proposition}

\begin{proof}
Since $H_0$ is invariant, $\pi$ is a homomorphism. Any compression map is completely contractive, so $\pi$ is a completely contractive homomorphism. Identifying $H_0\cong H$, on generators we have
\begin{equation}\label{eq:generators_mapped}
\pi(\tilde{E}_{ij})=E_{ij}.
\end{equation}
This implies that $\pi(A)$ is dense in $K(H_0)$.
For $n\ge 1$, let
\[
\A_n = \Span\{E_{ij}\mid 1\le i,j\le n\},
\]
so that $\A_n$ are closed subalgebras with $\A=\overline{\bigcup_n \A_n}$. The relation \eqref{eq:generators_mapped} shows $\pi$ is injective on each $A_n$. Suppose that $a\in \A$ satisfies $\pi(a)=0$. With
\[
u_n = \sum_{k=1}^n \tilde{E}_{kk},
\]
we have $u_nau_n\in \A_n$ and
\[
\pi(u_nau_n) = \pi(u_n)\pi(a)\pi(u_n)=0
\]
for all $n$. Since $\pi\vert_{\A_n}$ is injective, $u_nau_n=0$ for all $n$, and by Lemma \ref{lem:A_approx_identity}, $a=\lim_n u_nau_n=\lim_n0 = 0$. Therefore $\pi$ is injective.

By the inverse mapping theorem, if $\pi$ were onto, it would be bounded below. However, $\pi$ is not bounded below, for instance
\[
\|u_n\|=\|\tilde{E}_{nn}\| = \sqrt{1+n^2},
\]
while
\[
\|\pi(u_n)\|=\|E_{11}+\cdots+E_{nn}\|=1.
\]
So, $\pi$ is not onto.
\end{proof}

\begin{theorem}\label{thm:simple_not_cb_iso_to_C*}
The simple operator algebra $\A$ is not completely boundedly isomorphic to a C*-algebra.
\end{theorem}

\begin{proof}
Suppose for the sake of contradiction that there is a C*-algebra $\B$ and a completely bounded algebra isomorphism $\rho:\B\to \A$. By Proposition \ref{prop:A_restricts_to_K} there is a completely contractive injective homomorphism $\pi:\A\to K(H_0)$ with dense but not closed range. Then $\sigma:=\pi\rho:\B\to K(H_0)$ is a nondegenerate injective completely bounded homomorphism between C*-algebras with non-closed range.

By \cite[Theorem 1.10]{Haagerup}, any nondegenerate completely bounded homomorphism of C*-algebras is similar to a $\ast$-homomorphism. So, there is an invertible operator $S\in B(H)$ such that
\[
b\mapsto S\sigma(b) S^{-1}
\]
is a $\ast$-homomorphism. Since a $\ast$-homomorphism has closed range, $S\sigma(\B)S^{-1}$ is closed. But then
$
\sigma(B)=
S^{-1}(S\sigma(\B)S^{-1})S
$
is also closed, because $\ad_{S^{-1}}$ is a homeomorphism, and this contradicts the conclusion of the first paragraph above.
\end{proof}

In particular, $\A$ cannot be similar to a C*-algebra in any completely isometric representation, since a completely isometric isomorphism and a similarity compose to a completely bounded isomorphism.

\section*{Acknowledgements}
The authors would like to thank Terry Loring and Narutaka Ozawa for their helpful comments on Mathoverflow, to Matt Kennedy for pointing out the paper \cite{KirWas}, and Elias Katsoulis for asking whether all simple operator algebras are similar to C*-algebras. The second author was supported by the NSERC Discovery grant 2019-05430. The first author thanks Adi Tcaciuc, Nicolae Strungaru, and the second author for their support as a Postdoctoral Fellow.


\begin{thebibliography}{99}

\bibitem{Arv1} W. Arveson, \textit{Subalgebras of C$^*$-algebras}, Acta Mathematica {\bf 123} (1969), no. 1, 141–224.

\bibitem{Blech} D. Blecher, \textit{Modules over operator algebras and the maximal C$^*$-dilation}, J. Func. Anal. {\bf 169} (1999), 251-288.

\bibitem{Blecher} D. Blecher and C. Le Merdy, \textit{Operator algebras and their modules: an operator space approach}, Oxford University Press (2004), 30.

\bibitem{BrownOzawa}
N. Brown and N. Ozawa, \textit{C$^*$-algebras and finite-dimensional approximations}, Graduate Studies in Mathematics {\bf 88}, 2008, Providence, RI: American Mathematical Society.

\bibitem{CDO}
R. Clou\^atre and A. Dor-On, \textit{Finite-Dimensional Approximations and Semigroup Coactions for Operator Algebras}, International Mathematics Research Notices (2024), 698–744.

\bibitem{CR}
R. Clou\^atre and C. Ramsey, \textit{Residually finite-dimensional operator algebras}, J. Func. Anal. {\bf 277} (2019), 2572–2616.

\bibitem{DKDoc} K. Davidson and E. Katsoulis, \textit{Dilation theory, commutant lifting and semicrossed products}, Doc. Math. \textbf{16} (2011), 781--868.

\bibitem{DFK} K. Davidson, A. Fuller, and E. Kakariadis, \textit{Semicrossed products of operator algebras by semigroups}, Mem. Amer. Math. Soc \textbf{247}, (2017), no. 1168. 

\bibitem{DOEG} A. Dor-On, S. Eilers, and S. Geffen, \textit{Classification of irreversible and reversible Pimsner operator algebra} Compositio Mathematica {\bf 156} (2020), 2510-2535.

\bibitem{DOS} A. Dor-On and G. Salomon, \textit{Full Cuntz–Krieger dilations via non-commutative boundaries}, J. London Math. Soc. {\bf 98} (2018), 416-438. 

\bibitem{DriMcc}
M. Dritschel and S. McCullough, \textit{Boundary representations for families of representations of operator algebras and spaces}, J. Op. Thy. (2005), 159–167.

\bibitem{Grill}
W. Grilliette, \textit{Presentations and Tietze transformations of C$^*$-algebras}, New York J. Math. {\bf 18} (2012), 121-137.

\bibitem{Haagerup} 
U. Haagerup, \textit{Solution of the similarity problem for cyclic representations of C*-algebras}, Ann. of Math. {\bf 118} (1983), 215-240.

\bibitem{Hamana}
M. Hamana, \textit{Injective envelopes of operator systems}, Publications of the Research
Institute for Mathematical Sciences {\bf 15} (1979), no. 3, 773–785.

\bibitem{Hamidi}
M. Hamidi, \textit{Admissibility of C$^*$-covers and crossed products of operator algebras},
Ph.D. Thesis, \url{https://digitalcommons.unl.edu/mathstudent/95} (2019).

\bibitem{Hartz} M. Hartz, \textit{Finite dimensional approximations in operator algebras}, J. Func. Anal. {\bf 285} (2023), 109974.

\bibitem{Humen} A. Humeniuk, \textit{C*-envelopes of semicrossed products by lattice ordered abelian semigroups}, J. Funct. Anal. {\bf 279} (2020), 108731.

\bibitem{KatRamMem} E. Katsoulis and C. Ramsey, \textit{Crossed products of operator algebras},
Mem. Amer. Math. Soc {\bf 258} (2019), no. 1240.

\bibitem{KatRamHN} E. Katsoulis and C. Ramsey, \textit{The non-selfadjoint approach to the Hao-Ng isomorphism problem}, International Mathematics Research Notices (IMRN) (2021), 1160-1197.

\bibitem{KatRamHyper}
E. Katsoulis and C. Ramsey, \textit{The hyperrigidity of tensor algebras of C$^*$-correspondences}, J. Math Anal. App. {\bf 483} (2020), 123611, 10 pp.

\bibitem{KirWas}
E. Kirchberg and S. Wassermann, \textit{C$^*$-algebras generated by operator systems}, J. Func. Anal. {\bf 155} (1998), 324-351.

\bibitem{Loring}
T. Loring, \textit{C$^*$-Algebras Generated by Stable Relations}, J. Func. Anal. {\bf 112} (1993), 159-203.

\bibitem{Meyer}
R. Meyer, \textit{Adjoining a unit to an operator algebra}, J. Operator Theory {\bf 46} (2001), 281–288.

\bibitem{MS}
P. Muhly and B. Solel, \textit{An algebraic characterization of boundary representations}, Nonselfadjoint operator algebras, operator theory, and related topics, 189–196, Oper. Theory Adv. Appl. {\bf 104}, Birkhser, Basel, 1998.

\bibitem{MS1} 
P. Muhly and B. Solel, \textit{Tensor algebras over $C^*$-correspondences: representations, dilations, and $C^*$-envelopes},  J. Funct. Anal.  \textbf{158}  (1998), 389--457.

\bibitem{Paulsen} V. Paulsen, \textit{Completely bounded maps and operator algebras}, Cambridge University Press (2002), 78.



\bibitem{Takesaki} M. Takesaki, \textit{Theory of operator algebras I}, Vol. 124. Berlin: Springer, 2003.



\bibitem{Thompson} I. Thompson, \textit{Maximal C$^*$-covers and residual finite-dimensionality}, J. Math. Anal. App. {\bf 514} (2022), 126277.

\end{thebibliography}
\end{document}